\documentclass[final,onefignum,onetabnum]{siamart190516}

\usepackage{amsmath}
\RequirePackage[numbers]{natbib}
\RequirePackage{hyperref}
\usepackage[utf8]{inputenc} 
\usepackage{amsfonts}
\usepackage{bm}
\usepackage{cleveref}
\usepackage{graphicx}
\usepackage[export]{adjustbox}
\usepackage[caption=false,font=footnotesize,indention=0cm]{subfig}
\usepackage[shortlabels]{enumitem}
\usepackage{algorithm}
\usepackage{algpseudocode}

\theoremstyle{definition}
\newtheorem{example}[theorem]{Example}
\newtheorem{remark}[theorem]{Remark}

\def\N{\mathbb{N}}
\def\R{\mathbb{R}}
\def\C{\mathbb{C}}
\def\lmax{\lambda_{\mathrm{max}}}

\def\addlegendimage{\csname pgfplots@addlegendimage\endcsname}

\ifpdf
\hypersetup{
  pdftitle={Adaptive rational Krylov methods for exponential Runge--Kutta integrators},
  pdfauthor={K. Bergermann, and M. Stoll}
}
\fi

\headers{Rat.\ Krylov exponential Runge--Kutta integrators}{K. Bergermann, and M. Stoll}

\title{Adaptive rational Krylov methods for exponential Runge--Kutta integrators}

\author{Kai Bergermann\thanks{Chair of Scientific Computing, Technische Universit\"at Chemnitz, Department of Mathematics, 09107 Chemnitz, Germany
		(\email{kai.bergermann@math.tu-chemnitz.de}, \email{martin.stoll@math.tu-chemnitz.de}).}
	\and Martin Stoll\footnotemark[1]
}

\begin{document}

\maketitle

\begin{abstract}
We consider the solution of large stiff systems of ordinary differential equations with explicit exponential Runge--Kutta integrators. 
These problems arise from semi-discretized semi-linear parabolic partial differential equations on continuous domains or on inherently discrete graph domains.
A series of results reduces the requirement of computing linear combinations of $\varphi$-functions in exponential integrators to the approximation of the action of a smaller number of matrix exponentials on certain vectors.
State-of-the-art computational methods use polynomial Krylov subspaces of adaptive size for this task.
They have the drawback that the required number of Krylov subspace iterations to obtain a desired tolerance increase drastically with the spectral radius of the discrete linear differential operator, e.g., the problem size.
We present an approach that leverages rational Krylov subspace methods promising superior approximation qualities.
We prove a novel a-posteriori error estimate of rational Krylov approximations to the action of the matrix exponential on vectors for single time points, which allows for an adaptive approach similar to existing polynomial Krylov techniques.
We discuss pole selection and the efficient solution of the arising sequences of shifted linear systems by direct and preconditioned iterative solvers.
Numerical experiments show that our method outperforms the state of the art for sufficiently large spectral radii of the discrete linear differential operators.
The key to this are approximately constant numbers of rational Krylov iterations, which enable a near-linear scaling of the runtime with respect to the problem size.
\end{abstract}



\section{Introduction}\label{sec:intro}

The efficient numerical solution of ordinary differential equations (ODEs) is a fundamental problem in numerical analysis and a large body of work has been devoted to this problem, cf.\ e.g., \cite{hairer1991solving,lambert1991numerical}.
In this paper, we consider large and stiff systems of ODEs \cref{eq:ode_system} with discrete linear differential operators $\bm{A}\in\R^{n \times n}$ and semi-linear functions $g$.
Such problems arise, e.g., from semi-discretized semi-linear parabolic partial differential equations (PDEs) on continuous domains or on inherently discrete graph domains.
While the former problem with spatial discretizations by finite differences or finite elements is a very classical one, the simulation of dynamical processes on discrete graphs or networks (we use the two terms synonymously throughout the manuscript) has recently gained attention \cite{benzi2020matrix,estrada2021path,benzi2022rational,hutt2022predictable}.

In principle, a large variety of techniques such as Runge--Kutta methods are available for the numerical integration of systems of ODEs.
Recent decades, however, have witnessed an increased interest in exponential integrators, which are particularly well-suited for the solution of stiff or highly oscillatory problems \cite{cox2002exponential,hochbruck2005explicit,hochbruck2005exponential,krogstad2005generalized,tokman2006efficient,hochbruck2010exponential,tokman2011new,weiner2013linear,rainwater2016new}.
Exponential integrators owe their name to the matrix exponential propagator and have decisively fueled a successful line of research on the efficient approximation of the matrix exponential function \cite{moler1978nineteen,saad1992analysis,hochbruck1997krylov,sidje1998expokit,moler2003nineteen,moret2004rd,higham2005scaling,van2006preconditioning,higham2008functions,beckermann2009error,al2010new,al2011computing}.
In particular, we mention polynomial Krylov subspace methods \cite{saad1992analysis,hochbruck1997krylov,moler2003nineteen,higham2008functions,golub2013matrix} in combination with rational Pad\'e approximations of the compressed Hessenberg representation of the original matrix \cite{higham2005scaling,al2010new}.
In the field of network science, these techniques also provide a variety of insights into structural network properties \cite{estrada2005subgraph,estrada2010network,benzi2020matrix,bergermann2022fast}.

In addition to the matrix exponential, exponential integrators generally require the evaluation of linear combinations of $\varphi$-functions acting on certain vectors that depend on the trajectory of the system of ODEs.
A series of results by Saad \cite{saad1992analysis}, Sidje \cite{sidje1998expokit}, and Al-Mohy and Higham \cite{al2011computing} shows that this problem can be reduced to the computation of the action of the matrix exponential of a matrix $\widetilde{\bm{A}}$ on vectors $\widetilde{\bm{c}}$, where $\widetilde{\bm{A}}$ is a slightly enlarged version of $\bm{A}$.
The computational efficiency of exponential integrators is thus determined by the efficiency of computing quantities $e^{h_i\widetilde{\bm{A}}}\widetilde{\bm{c}}$ for given time step sizes $h_i\in\R_{>0}$.

The direct approximation of the matrix exponential \cite{moler1978nineteen,moler2003nineteen,higham2005scaling,al2010new} is computationally burdensome in terms of runtime and memory requirement and hence infeasible for medium to large problem sizes.
Since only its action on vectors is required, state-of-the-art software packages \texttt{phipm} \cite{niesen2012algorithm} and \texttt{KIOPS} \cite{gaudreault2018kiops} use polynomial Krylov subspace approximations.
These routines are adaptive in the sense that the polynomial Krylov subspace size and possibly a sub-time interval step size is chosen based on an a-posteriori error estimate of the approximation of $e^{h_i\widetilde{\bm{A}}}\widetilde{\bm{c}}$ \cite{saad1992analysis}.

While these methods are extremely effective for many problems, research efforts for their improvement are still ongoing \cite{deka2022comparison,croci2023exploiting,deka2023lexint}.
The major drawback is that the required number of polynomial Krylov iterations increases with $\|h_i\widetilde{\bm{A}}\|_2$, i.e., the time step size as well as the spectral radius of the discrete linear differential operator $\bm{A}$ that often behaves proportionally to the problem size, cf.\ \Cref{sec:odes}.
Since for a diagonalizable matrix the computation of the matrix exponential is equivalent to exponentiating its eigenvalues \cite{higham2008functions}, the oscillatory properties of polynomial approximations demand higher polynomial degrees when the approximation interval is increased \cite{trefethen2019approximation}.
This prevents a linear scaling of the runtime of polynomial Krylov subspace methods with respect to the problem size, cf.\ \Cref{fig:iter_runtime_2D_AC}.
The study of uniform rational (best) approximations of $e^{-x}$ on the unbounded positive (or equivalently, $e^x$ on the negative) semi-axis \cite{cody1969chebyshev,carpenter1984extended,gallopoulos1992efficient,trefethen2019approximation} instead promises the convergence of rational approximations independent of the length of positive approximation intervals.

The above can be viewed as one reason that rational Krylov subspace methods, which the improved approximation quality of rational functions is built into, have been studied intensively in recent years, cf.\ e.g., \cite{ruhe1984rational,ruhe1994rational,moret2004rd,van2006preconditioning,popolizio2008acceleration,guttel2013rational,gockler2013convergence,gander2013paraexp,zhuang2013power,gockler2014uniform,ragni2014rational,wang2017exploring,bertaccini2021computing}.
The space of rational functions representable by a rational Krylov subspace $\mathcal{Q}_m(\widetilde{\bm{A}}, \widetilde{\bm{c}})$ crucially depends on the choice of poles \mbox{$\xi_j\in\mathbb{C}\cup\{\infty\}, j=1, \dots , m-1$} and optimal pole selection strategies remain an active field of research, cf.\ e.g., \cite{druskin2011adaptive,guttel2013rational,borner2015three,berljafa2017rkfit,massei2021rational}.
The improved approximation quality of $\mathcal{Q}_m$, however, comes at the cost of the requirement of a linear system solve in each iteration that crucially affects the computational efficiency of rational Krylov methods.
Depending on the choice of poles, the encountered sequence of shifted linear systems may contain complex-valued and indefinite problems, which complicates their efficient solution.
Besides direct approaches based on the LU or Cholesky decomposition for admissible problem sizes \cite{golub2013matrix}, we use iterative solvers \cite{saad2003iterative} preconditioned with algebraic multigrid \cite{ruge1987algebraic,falgout2006introduction,notay2010aggregation,notay2012aggregation,napov2012algebraic}, which allows a near-linear scaling of the runtime in the problem size.

A popular, computationally less demanding special case of rational Krylov subspaces are shift \& invert Krylov subspaces  \cite{moret2004rd,van2006preconditioning}, which use only one single repeated pole.
These methods have been leveraged extensively in recent years for matrix function approximation in general and exponential integration in particular \cite{popolizio2008acceleration,gockler2013convergence,gander2013paraexp,zhuang2013power,gockler2014uniform,ragni2014rational,wang2017exploring}.

In this work, we combine the power of rational Krylov subspace methods with multiple complex-valued poles with the efficient implementation of exponential integrators as well as the adaptive Krylov subspace sizes used in state-of-the-art methods for exponential integration.
Our strategy is made possible by a novel a-posteriori error estimate to rational Krylov subspace approximations to $e^{h_i\widetilde{\bm{A}}}\widetilde{\bm{c}}$ presented in \Cref{prop:rat_krylov_error}.
Computational efficiency is ensured by optimal pole selection and linear system solving techniques discussed in \Cref{sec:rk_poles,sec:rk_linear_system_solves}, respectively.
\Cref{fig:iter_runtime_2D_AC} illustrates both the main motivation and the main contribution of this paper, namely approximately constant rational Krylov subspace iteration numbers leading to a near-linear scaling of the runtime for the solution of large stiff systems of ODEs.
We implement our method in the routine $(RK)^2$EXPINT (Rational Krylov Runge--Kutta exponential integrators, \texttt{rk2expint}) and Matlab codes are publicly available under \url{https://github.com/KBergermann/rk2expint}.

We test our method on two semi-linear parabolic PDEs:
the Allen--Cahn and Gierer--Meinhardt equations.
As discrete linear differential operators we choose finite difference discretizations of the two-dimensional continuous Laplacian operator as well as the (unnormalized) graph Laplacian of inherently discrete network domains.
Numerical experiments show that $(RK)^2$EXPINT is capable of outperforming state-of-the-art methodology for sufficiently large spectral radii of the discrete linear differential operators, i.e., large problem sizes or large time step sizes.

\begin{figure}
	\centering
	\subfloat[Krylov iteration numbers]{
		\includegraphics[width=.43\textwidth]{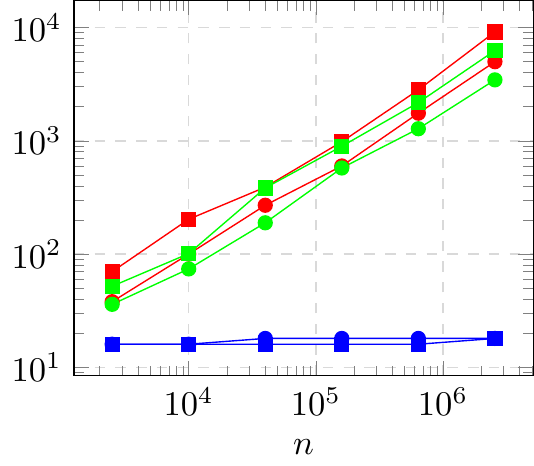}
	}
	\hfill
	\subfloat[Runtime in seconds]{
		\includegraphics[width=.43\textwidth]{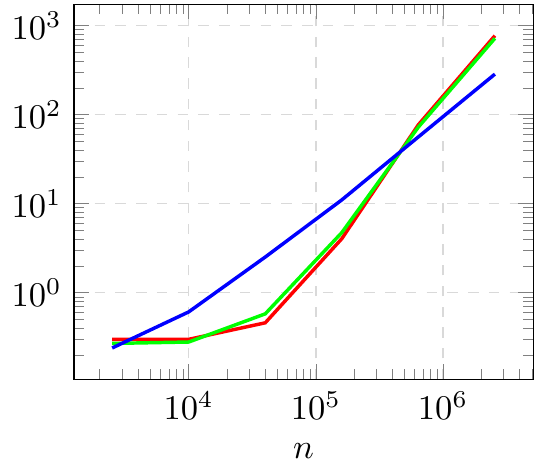}
	}
	\vspace{10pt}
	\includegraphics[width=.8\textwidth]{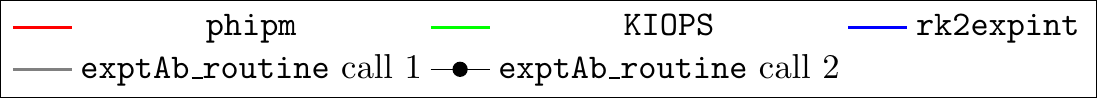}
	\vspace{-10pt}
	\caption{Comparison of polynomial (\texttt{phipm}, \texttt{KIOPS}) and rational (\texttt{rk2expint}) Krylov iteration numbers per evaluation of $e^{h_i \widetilde{\bm{A}}}\widetilde{\bm{c}}$ and total runtimes of solving the 2D Allen--Cahn equation with homogeneous Neumann boundary conditions.
	The total number of grid points is denoted by $n$.
	We use the \texttt{SW2} exponential Runge--Kutta integrator, which requires two evaluations of $e^{h_i \widetilde{\bm{A}}}\widetilde{\bm{c}}$ per time step.
	For \texttt{rk2expint}, we use complex-valued $(35,30)$ \texttt{RKFIT} poles fitted on the interval $[0,10^6]$ as well as preconditioned iterative linear system solves.}\label{fig:iter_runtime_2D_AC}
\end{figure}

The rest of this paper is organized as follows.
\Cref{sec:odes} provides details on the large stiff systems of ODEs.
In \Cref{sec:expint}, we briefly introduce exponential Runge--Kutta integrators with a special focus on their efficient numerical implementation in \Cref{sec:expint_implementation}.
\Cref{sec:rk} starts by introducing rational Krylov subspace methods before discussing our choice of pole selection and the solution of the sequence of shifted linear systems.
In \Cref{sec:rk_error}, we prove our novel a-posteriori rational Krylov error estimate to quantities $e^{h_i\widetilde{\bm{A}}}\widetilde{\bm{c}}$ before summarizing our proposed method in an algorithm in \Cref{sec:algorithm}.
Finally, \Cref{sec:experiments} presents numerical experiments.

\paragraph{Notation}
To be consistent with as much notation from the literature as possible, we made the following notational choices:
we denote our discrete linear differential operators by the symmetric positive semi-definite matrix $\bm{A}$, the slightly enlarged version of $\bm{A}$ by the nonsymmetric negative semi-definite matrix $\widetilde{\bm{A}}$, and the poles $\xi_j$ are chosen to approximate the function $e^x$ on the negative real semi-axis.
The latter choice deviates from the standard notation that approximates $e^{-x}$ on the positive real semi-axis, but could not be avoided.

\section{Stiff systems of ODEs}\label{sec:odes}

In this paper, we consider the solution of large and stiff systems of ODEs of the form
\begin{equation}\label{eq:ode_system}
\frac{\partial \bm{u}(t)}{\partial t} = F(t,\bm{u}(t)) =  - \bm{A} \bm{u}(t) + g(t, \bm{u}(t)), \quad \bm{u}(0) = \bm{u}_0,
\end{equation}
with $\bm{u}\colon [0,T] \mapsto \R^n$ the desired solution on the time interval $[0,T]$, $F: [0,T] \times \R^n \rightarrow \R^n$ the right-hand side, $\bm{A}\in\R^{n \times n}$ a discrete linear differential operator, and $g: [0,T] \times \R^n \rightarrow \R^n$ a semi-linear function, i.e., $g$ is generally non-linear in $\bm{u}$ but contains no derivatives of $\bm{u}$.
The initial condition $\bm{u}_0\in\R^n$ together with appropriate boundary conditions built into $\bm{A}$ complete the initial boundary value problem.
In this paper, we restrict ourselves to symmetric positive semi-definite discrete linear differential operators $\bm{A}$.

Problems of the form \cref{eq:ode_system} arise in a multitude of applications including semi-discretized semi-linear parabolic PDEs
\begin{equation}\label{eq:pde}
\frac{\partial u(t,\bm{x})}{\partial t} = - \mathcal{A} u(t,\bm{x}) + g(t, u(t,\bm{x})), \quad u(0,\bm{x}) = u_0(\bm{x}),
\end{equation}
where $u: [0,T] \times \Omega \rightarrow \R$ is defined on the spatial domain $\Omega \subset \R^d$, $[0,T]$ denotes the time interval, $\mathcal{A}$ a linear differential operator between and $g$ a semi-linear function from appropriate function spaces.
In this paper, we restrict ourselves to the Laplacian operator $\mathcal{A}=-\Delta$ and its standard finite difference discretization.
\begin{definition}[\cite{strang2006linear}]\label{def:finite_differences}
The real symmetric finite difference matrix of an equispaced triangulation of a spatial interval of length $L_x$ by $n_x$ grid points, i.e., spatial step size $h_x = \frac{L_x}{n_x}$ is defined as
\begin{equation}\label{eq:fd_matrix}
\bm{T}_{n_x} = \frac{1}{h_x^2}\mathrm{tridiag}(-1,2,-1)
\in\R^{n_x \times n_x}.
\end{equation}
Dirichlet, Neumann, or periodic boundary conditions can be built into $\bm{T}_{n_x}$ by slight modification of the first and last row.
With this, we obtain the finite difference discretization of the two-dimensional continuous Laplacian operator as
\begin{equation*}
 \bm{A}=\bm{T}_{n_x} \otimes \bm{I} + \bm{I} \otimes \bm{T}_{n_x}.
\end{equation*}
\end{definition}
Finite difference matrices have been studied intensively over past decades and the full eigendecomposition $\bm{T}_{n_x} = \bm{\Phi}\bm{\Lambda}\bm{\Phi}^T$ of \cref{eq:fd_matrix} is known analytically, allowing efficient solution strategies, e.g., based on fast Fourier or discrete cosine transforms.

\begin{proposition}[\cite{strang2006linear,golub2013matrix}]\label{prop:fd_spectra}
The spectrum of $\bm{T}_{n_x}$ is contained in the real interval \mbox{$\Sigma = \frac{1}{h_x^2}[0,4]$} for all boundary conditions.
Due to the properties of the Kronecker product, the spectrum of the two-dimensional finite difference Laplacian is contained in $\Sigma=\frac{2}{h_x^2}[0,4]$.
\end{proposition}

Although not explored in this paper, we remark that all methods should equally apply to finite element discretizations due to similar properties of the stiffness matrix.

The second application of interest to this paper is the simulation of dynamical processes on graphs/networks.

\begin{definition}
A graph $\mathcal{G} = (\mathcal{V}, \mathcal{E})$ is defined by a node set $\mathcal{V}$ with $|\mathcal{V}|=n$ and an edge set $\mathcal{E} \subset \mathcal{V} \times \mathcal{V}$.
We consider undirected and possibly weighted edges leading to the graph's symmetric adjacency matrix $\bm{W}\in\R_{\geq 0}^{n \times n}$ with
\begin{equation*}
\bm{W}_{ij} = 
\begin{cases}
w_{ij}>0 & \text{if there is an edge between nodes $i$ and $j$,}\\
0 & \text{otherwise,}
\end{cases}
\end{equation*}
for $i,j=1, \dots , n$.
Furthermore, we define the diagonal degree matrix $\bm{D} = \text{diag}(\bm{W}\bm{1})$ with $\bm{1}\in\R^n$ the vector of all ones.
Then, the (unnormalized) graph Laplacian operator is given by
\begin{equation*}
\bm{A} = \bm{L} = \bm{D} - \bm{W}.
\end{equation*}
\end{definition}

Note that with appropriately chosen edge weights, the graph Laplacian $\bm{L}$ coincides with the $d$-dimensional finite difference Laplacian on graphs representing $d$-dimensional regular grids.
We know from spectral graph theory that the spectrum of $\bm{L}$ for unweighted graphs is given by $\Sigma=[0, n]$ \cite{chung1997spectral}.
Positive weights preserve the positive semi-definiteness of $\bm{L}$, leading to the following summarizing remark.

\begin{remark}\label{rem:Aspsd}
The discrete linear differential operators $\bm{A}$ considered in this paper are symmetric positive semi-definite M-matrices, cf.~\cite[Chapter 6]{berman1994nonnegative}.
\end{remark}

\begin{remark}
	Throughout this manuscript, $n$ denotes the matrix size of the discrete linear differential operator $\bm{A}$, i.e., the total number of grid points for finite difference discretizations or the number of graph nodes.
	In the finite difference case, $n_x$ denotes the number of grid points in each spatial coordinate direction.
\end{remark}

The solution of problems of the form \cref{eq:ode_system} is often complicated by stiffness -- a phenomenon of differential equations that is typically identified with characteristics such as a large stiffness ratio, i.e., a large ratio of the absolute values of the largest and smallest eigenvalue of $\bm{A}$, different decay ratios of components of the solution, or the fact that implicit numerical time integration methods work much better than explicit ones \cite{lambert1991numerical}.

\section{Exponential Runge--Kutta integrators}\label{sec:expint}

Ideas for the numerical solution of differential equations date back at least to Euler $250$ years ago and nowadays a plethora of well-studied numerical time integration techniques is available in the literature, cf.\ e.g., \cite{hairer1991solving,lambert1991numerical}.
Runge--Kutta methods rank among the most popular such techniques.
Due to favourable stability properties it is well-known that implicit Runge--Kutta methods are much better suited for solving stiff ODEs than explicit ones.

In this paper, however, we consider exponential integrators, which owe their name to the matrix exponential propagator $e^{-t\bm{A}}$ that solves the homogeneous equation \cref{eq:ode_system}, i.e., $\frac{\partial \bm{u}(t)}{\partial t} = - \bm{A} \bm{u}(t)$ exactly for all $t\in[0,T]$ via the matrix-vector product $\bm{u}(t) = e^{-t\bm{A}}\bm{u}_0$.
The uniform boundedness and the capability of the exponential propagator to fully resolve linear oscillations makes exponential integrators successful methods for the solution of stiff and highly oscillatory systems of ODEs \cite{hochbruck2010exponential}.
In particular, we choose the class of explicit exponential Runge--Kutta methods, which is designed for problems with a natural splitting of the right-hand side of \cref{eq:ode_system} into linear and non-linear part \cite{hochbruck2005explicit,hochbruck2005exponential}.
For more general problems with general right-hand sides $F(t,\bm{u}(t))$, other approaches such as exponential Rosenbrock \cite{hochbruck2010exponential} or exponential propagation iterative Runge--Kutta (EPIRK) methods \cite{tokman2006efficient,tokman2011new} have been proposed, which obtain the splitting in \cref{eq:ode_system} by local linearizations of $F(t,\bm{u}(t))$ along the trajectory of the solution $\bm{u}(t)$.

The construction of explicit exponential Runge--Kutta integrators relies on the variation-of-constants formula
\begin{equation}\label{eq:expint_time_step}
\bm{u}_{i+1} := \bm{u}(t+h_i) = e^{-h_i \bm{A}}\bm{u}(t) + \int_{0}^{h_i} e^{-(h_i - \tau) \bm{A}} g(t+\tau,\bm{u}(t+\tau)) d\tau,
\end{equation}
which can be interpreted as integrating the linear part of \cref{eq:ode_system} on the time interval $[t, t+h_i]$ exactly and separately approximating the remainder integral by exponential quadrature.
Assuming $\bm{g}=g(\tau,\bm{u}(t))$ constant leads to the exponential Euler method, which involves the function $\varphi_1(z) = \frac{e^z - 1}{z}$ \cite{hochbruck2010exponential}.
More sophisticated exponential quadrature rules lead to schemes including further $\varphi$-functions.

\begin{definition}\label{def:phi_functions}
The $(k+1)$st $\varphi$-function is defined via the power series
\begin{equation*}
\varphi_{k+1}(z) = \sum_{j=0}^\infty \frac{z^j}{(j+k+1)!},
\end{equation*}
or the recurrence relation 
\begin{equation*}
\varphi_{k+1}(z) = \frac{\varphi_{k}(z)-\varphi_{k}(0)}{z}, \quad \varphi_{k}(0) = 1/k!, \quad \varphi_0(z)=e^z
\end{equation*}
\end{definition}

One can now employ the idea of Runge--Kutta methods and introduce $s$ internal stages $t+c_1 h_i, \dots , t+c_s h_i$ with $c_j\in[0,1]$ for $1 \leq j \leq s$ into the time interval $[t,t+h_i]$ leading to schemes of the form
\begin{align}
\bm{u}_{i+1} & = \chi(-h_i\bm{A})\bm{u}_i + h_i \sum_{j=1}^s b_j(-h_i \bm{A})\bm{G}_{ij}.\label{eq:exprk1}\\
\bm{U}_{ij} & = \chi_j(-h_i\bm{A})\bm{u}_i + h_i \sum_{k=1}^s a_{jk}(-h_i \bm{A})\bm{G}_{ik},\label{eq:exprk2}\\
\bm{G}_{ik} & = g(t_i + c_k h_i, \bm{U}_{ik}),\label{eq:exprk3}
\end{align}
where $\chi$, $\chi_j$, $a_{jk}$, and $b_j$ are $\varphi$-functions.
Choosing $\chi(z)=e^z$ and $\chi_j(z)=e^{c_j z}$, one can derive stiff order conditions that allow the construction of exponential integrators with a convergence order independent of the problem's stiffness \cite{hochbruck2005explicit,hochbruck2005exponential}.
Note that the classical convergence order is an upper bound to the stiff order.
As in Runge--Kutta methods, one can use Butcher tableaus to define a given integrator, cf.\ \Cref{tab:butcher_tableau}.
An example of the stage $3$, stiff order $3$ method \texttt{ETD3RK} \cite{cox2002exponential} is given in \Cref{tab:butcher_tableau_ETD3RK}.
Here, $\varphi_{j,k}=\varphi_j(-c_k h_i \bm{A})$ and $\varphi_{j}=\varphi_j(-h_i \bm{A})$.
For details on the construction and analysis of explicit exponential Runge--Kutta integrators, we refer to \cite{hochbruck2005explicit,hochbruck2005exponential,hochbruck2010exponential} and references therein.

\begin{table}
\centering
\begin{tabular}{c|cccc}
$c_1$ &&&&\\
$c_2$ & $a_{2,1}(-h_i \bm{A})$ &&&\\
\vdots & \vdots & $\ddots$ &&\\
$c_s$ & $a_{s,1}(-h_i \bm{A})$ & $\hdots$ & $a_{s,s-1}(-h_i \bm{A})$&\\\hline
& $b_1(-h_i \bm{A})$ & $\hdots$ & $b_{s-1}(-h_i \bm{A})$ & $b_s(-h_i \bm{A})$
\end{tabular}
\caption{Butcher tableau of a general stage $s$ explicit exponential Runge--Kutta integrator.}\label{tab:butcher_tableau}
\end{table}

\begin{table}
	\centering
	\begin{tabular}{c|ccc}
		$0$ &&&\\
		$\frac{1}{2}$ & $\frac{1}{2}\varphi_{1,2}$&&\\
		$1$ & $-\varphi_{1,3}$ & $2\varphi_{1,3}$&\\\hline
		& $4\varphi_3 - 3\varphi_2+\varphi_1$ & $-8\varphi_3 + 4\varphi_2$ & $4\varphi_3 - \varphi_2$
	\end{tabular}
	\caption{Butcher tableau of the stage $3$, stiff order $3$ integrator \texttt{ETD3RK} \cite{cox2002exponential}.}\label{tab:butcher_tableau_ETD3RK}
\end{table}

\subsection{Efficient implementation}\label{sec:expint_implementation}

\Cref{eq:exprk1,eq:exprk2,eq:exprk3} show that each time step of an exponential integrator requires the evaluation of linear combinations of the action of $\varphi$-functions on vectors that depend on the trajectory of the ODE solution.
For the problem of computing $f(\bm{A})\bm{b}$, highly efficient methods based on Krylov subspace methods are available \cite{higham2008functions,golub2013matrix}.
They are based on constructing an orthonormal basis $\bm{V}_m$ of the polynomial Krylov subspace
\begin{equation*}
\mathcal{K}_m(\bm{A},\bm{b}) = \text{span}\{\bm{b}, \bm{A}\bm{b}, \dots , \bm{A}^{m-1}\bm{b}\}
\end{equation*}
leading to the approximation
\begin{equation}\label{eq:fAb_polynomial}
f(\bm{A})\bm{b} \approx \|\bm{b}\|_2 \bm{V}_m f(\bm{H}_m) \bm{e}_1,
\end{equation}
where $\bm{H}_m\in\R^{m \times m}$ is the Hessenberg reduction of $\bm{A}$ in $\mathcal{K}_m(\bm{A},\bm{b})$ and $\bm{e}_1\in\R^m$ denotes the first unit vector.
Equality in \cref{eq:fAb_polynomial} holds if $m$ is greater or equal to the invariance index of $\mathcal{K}_m$.
The main computational cost of such methods are matrix-vector products with $\bm{A}$ and applied to each individual $\varphi$-function this approach still proves computationally burdensome as, e.g., \texttt{ETD3RK} defined in \Cref{tab:butcher_tableau_ETD3RK} would require the computation of $10$ such quantities per time step.

A series of results by Saad \cite[Proposition 2.1]{saad1992analysis}, Sidje \cite[Theorem 1]{sidje1998expokit}, and Al-Mohy and Higham \cite[Theorem 2.1]{al2011computing} shows that the task can be reduced to the approximation of fewer quantities of the form $f(\bm{A})\bm{b}$.
We restate the special case $l=0$ of \cite[Theorem  2.1]{al2011computing} relevant to our problem in the notation defined above and formulated for the more general complex-valued case.

\begin{theorem}[Al-Mohy, Higham \cite{al2011computing}]\label{thm:al-mohy_higham}
	Let $\widetilde{\bm{A}} = \begin{pmatrix}
	-\bm{A} & \bm{C}\\
	\bm{0} & \bm{J}_p
	\end{pmatrix}\in\C^{(n+p) \times (n+p)}$,
	where $\bm{C}=[\bm{c}_p, \dots , \bm{c}_1]\in\C^{n \times p}$ and $\bm{J}_p\in\C^{p \times p}$ a Jordan block to the eigenvalue $0$.
	Furthermore, we define the matrix exponential $\bm{X}=e^{h_i \widetilde{\bm{A}}}$ as well as the vector $\widetilde{\bm{c}} := \begin{pmatrix}
	\bm{c}_0\\
	\bm{e}_p
	\end{pmatrix}\in\C^{n+p}$.
	Then, we have \mbox{$\bm{X}(1:n, n+p) = \sum_{k=1}^p h_i^k \varphi_k(-h_i \bm{A})\bm{c}_k$} and
	\begin{equation*}
	\bm{X} \widetilde{\bm{c}}
	= e^{h_i\widetilde{\bm{A}}} \widetilde{\bm{c}}
	= \begin{pmatrix}
	\sum_{k=0}^p h_i^k \varphi_k(-h_i \bm{A})\bm{c}_k\\
	e^{\bm{J}_p}\bm{e}_p
	\end{pmatrix}
	:=\widetilde{\bm{b}}\in\C^{n+p}.
	\end{equation*}
\end{theorem}

 \begin{remark}
 Since $\widetilde{\bm{A}}$ defined in \Cref{thm:al-mohy_higham} is upper block triangular, its spectrum is the union of the spectrum of $-\bm{A}$ with the eigenvalue $0$ with multiplicity $p$ independently of the matrix $\bm{C}$, making $\widetilde{\bm{A}}$ negative semi-definite.
 \end{remark}

 With \Cref{thm:al-mohy_higham}, the task for a given exponential Runge--Kutta integrator becomes grouping the terms from \cref{eq:exprk1,eq:exprk2,eq:exprk3} such that all required linear combinations of $\varphi$-functions can be obtained by as few quantities $e^{h_i\widetilde{\bm{A}}} \widetilde{\bm{c}}$ as possible.
 
 This idea has been exploited in the software package \texttt{phipm} \cite{niesen2012algorithm} as well as a later package \texttt{KIOPS} \cite{gaudreault2018kiops}, which provides a number of modifications to \texttt{phipm}.
 The common idea of both packages is to apply polynomial Krylov subspace methods discussed above to obtain approximations
 \begin{equation}\label{eq:polynomial_krylov_approximation}
  e^{h_i\widetilde{\bm{A}}} \widetilde{\bm{c}} \approx \|\widetilde{\bm{c}}\|_2 \bm{V}_m e^{h_i\bm{H}_m} \bm{e}_1.
 \end{equation}
 The matrix exponential of the small Hessenberg matrix $\bm{H}_m$ can be computed efficiently by various means \cite{moler1978nineteen,moler2003nineteen}, with the current Matlab standard\footnote{as implemented in the \texttt{expm} function in Matlab version R2020b} being rational Pad\'e approximations computed by the scaling and squaring algorithm \cite{higham2005scaling,al2010new}.

 Both \texttt{phipm} and \texttt{KIOPS} approximate \cref{eq:polynomial_krylov_approximation} to a user-specified tolerance in an adaptive way.
 The adaptivity relies on an a-posteriori error estimate to \cref{eq:polynomial_krylov_approximation} proposed by Saad \cite[Theorem 5.1]{saad1992analysis} who proved the first version of \Cref{thm:al-mohy_higham} not in the context of exponential integration but of analyzing polynomial Krylov subspace approximations to the action of the matrix exponential on vectors.
 If \cref{eq:polynomial_krylov_approximation} does not yet meet the tolerance, the approximation can be improved by either increasing the polynomial Krylov subspace size $m$ or sub-stepping the time interval $[0, h_i]$.
 The sub-stepping is motivated by interpreting $e^{h_i\widetilde{\bm{A}}} \widetilde{\bm{c}}$ as the solution to the differential equation
 \begin{equation*}
 \frac{\partial \bm{u}}{\partial t} = \widetilde{\bm{A}}\bm{u}, \quad \bm{u}_0=\widetilde{\bm{c}},
 \end{equation*}
 on the time interval $[0, h_i]$ and realizing that
 \begin{equation}\label{eq:time_interval_substepping}
  e^{h_i\widetilde{\bm{A}}} \widetilde{\bm{c}} = e^{(h_{i1}+h_{i2}+\dots+h_{ik})\widetilde{\bm{A}}} \widetilde{\bm{c}} = e^{h_{ik}\widetilde{\bm{A}}}(\cdots(e^{h_{i2}\widetilde{\bm{A}}}(e^{h_{i1}\widetilde{\bm{A}}}\widetilde{\bm{c}}))),
 \end{equation}
 where $h_i=h_{i1}+h_{i2}+\dots+h_{ik}$ for $h_{ij}>0$ and $j=1, \dots , k$.
 
 The polynomial Krylov approximation of the action of $e^{h_{ij}\widetilde{\bm{A}}}$ on a vector to a given tolerance can be achieved with a lower polynomial degree if $\|h_{ij}\widetilde{\bm{A}}\|_2 < \|h_{i}\widetilde{\bm{A}}\|_2$ for $h_{ij}<h_i$, cf.\ the discussion in \Cref{sec:intro}.
 We refer to \cite{niesen2012algorithm,gaudreault2018kiops} for details on how the adaptivity is implemented in \texttt{phipm} and \texttt{KIOPS} with the goal of minimizing the number of matrix-vector products.
 
 \begin{sloppypar}
 As already mentioned at the beginning of \Cref{sec:expint}, alternative approaches to exponential Runge--Kutta methods are given by exponential Rosenbrock \cite{hochbruck2010exponential} or EPIRK methods \cite{tokman2006efficient,tokman2011new}, which rely on local linearizations of general right-hand sides $F(t,\bm{u}(t))$ in \cref{eq:ode_system}.
 These methods hold the potential to be computationally more efficient than exponential Runge--Kutta methods.
 In particular, the \texttt{KIOPS} package allows the evaluation of EPIRK methods by fewer quantities of the form $e^{h_i\widetilde{\bm{A}}} \widetilde{\bm{c}}$ than \texttt{phipm} by differentiating between two tasks:
 task $1$ leverages \cref{eq:time_interval_substepping} to allow the approximation of several vectors $\varphi_k(-h_{ij}\bm{A})\bm{c}_k$ for one fixed $k$ at different time points $h_{ij}$;
 task $2$ addresses the computation of linear combinations of multiple $\varphi$-functions $\sum_{k=0}^p h_i^k \varphi_k(-h_i\bm{A})\bm{c}_k$ with $h_i=1$.
 It therefore appears attractive to combine the approach presented in this work with these types of integrators.
 The complication, however, is that due to the dependence of the linearization of the general right-hand side $F(t,\bm{u}(t))$ on the trajectory $\bm{u}(t)$ the matrix $\widetilde{\bm{A}}$ is generally different in each time step, allowing no statements on the spectrum or the definiteness property of $\widetilde{\bm{A}}$ similar to \Cref{rem:Aspsd}.
 We leave this question to future research.
 \end{sloppypar}

\section{Rational Krylov subspace methods}\label{sec:rk}

The current state-of-the-art methods for exponential integrators discussed in \Cref{sec:expint_implementation} are based on representing the matrix $\widetilde{\bm{A}}$ in a polynomial Krylov subspace in order to then apply a cheap rational Pad\'e approximation to its compression.
In this section, we review rational Krylov subspace methods for which the rational approximation is built into the Krylov space \cite{ruhe1984rational,ruhe1994rational,guttel2013rational}.
Results from approximation theory attest the superior quality of the approximation of the exponential function by rational functions in comparison to polynomials, cf.\ the discussion in \Cref{sec:intro}.
We state the results in this section for complex-valued matrices in the notation of the previous section, i.e., $\widetilde{\bm{A}}\in\C^{(n+p) \times (n+p)}$ and vectors $\widetilde{\bm{c}}\in\C^{(n+p)}$.

\begin{definition}[\cite{guttel2013rational}]
The rational Krylov subspace of size $m$ of a matrix $\widetilde{\bm{A}}$ and a vector $\widetilde{\bm{c}}$ is defined as
\begin{equation*}
\mathcal{Q}_m(\widetilde{\bm{A}}, \widetilde{\bm{c}}) = q_{m-1}(\widetilde{\bm{A}})^{-1} \text{span} \{\widetilde{\bm{c}}, \widetilde{\bm{A}}\widetilde{\bm{c}}, \dots , \widetilde{\bm{A}}^{m-1}\widetilde{\bm{c}}\} = q_{m-1}(\widetilde{\bm{A}})^{-1} \mathcal{K}_m(\widetilde{\bm{A}}, \widetilde{\bm{c}}),
\end{equation*}
where $q_{m-1}$ denotes the denominator polynomial, which we assume to be factored, i.e.,
\begin{equation*}
q_{m-1}(z) = \prod_{j=1}^{m-1} (1 - z/\xi_j).
\end{equation*}
The scalars $\xi_j\in\mathbb{C}\cup\{\infty\}, j=1, \dots , m-1$ denote the poles of $\mathcal{Q}_m$, which must not coincide with eigenvalues of $\widetilde{\bm{A}}$ to ensure the invertibility of $q_{m-1}$.
\end{definition}

For $\bm{V}_m$ an orthonormal basis of $\mathcal{Q}_m(\widetilde{\bm{A}}, \widetilde{\bm{c}})$, the rational Arnoldi relation reads
\begin{equation}\label{eq:rational_arnoldi_decomposition}
\widetilde{\bm{A}} \bm{V}_{m+1} \underline{\bm{K}_m} = \bm{V}_{m+1} \underline{\bm{H}_m}
\end{equation}
with
\begin{equation*}
\underline{\bm{H}_m} = \begin{pmatrix}
\bm{H}_m\\
h_{m+1,m}\bm{e}_m^\ast
\end{pmatrix},
\quad
\underline{\bm{K}_m} = \begin{pmatrix}
\bm{I}_m + \bm{H}_m\bm{D}_m\\
h_{m+1,m}\xi_m^{-1}\bm{e}_m^\ast
\end{pmatrix},
\end{equation*}
where $\bm{D}_m = \text{diag}(\xi_1^{-1}, \dots , \xi_m^{-1})$ and $\bm{H}_m,\bm{K}_m\in\C^{m \times m}$.
Furthermore, $h_{m+1,m}$ denotes the usual norm of the residual vector \cite{golub2013matrix} that should not be confused with the time step size $h_i$.
Note that the special case $\xi_1 = \dots = \xi_{m-1} = \infty$ recovers the polynomial Krylov subspace, whereas the case $\xi_1 = \dots = \xi_{m-1} \neq \infty$ is referred to as the shift \& invert Krylov subspace \cite{moret2004rd,van2006preconditioning}.

We continue by slightly rewriting the rational Arnoldi decomposition \cref{eq:rational_arnoldi_decomposition}:
\begin{align*}
&~\widetilde{\bm{A}}\bm{V}_m (\bm{I} + \bm{H}_m\bm{D}_m) + \widetilde{\bm{A}}\bm{v}_{m+1} h_{m+1,m} \xi_m^{-1}\bm{e}_m^\ast = \bm{V}_m \bm{H}_m + \bm{v}_{m+1} h_{m+1,m} \bm{e}_m^\ast\\
\Leftrightarrow &~\widetilde{\bm{A}}\bm{V}_m\bm{K}_m = \bm{V}_m \bm{H}_m + h_{m+1,m} (\bm{I} - \xi_m^{-1}\widetilde{\bm{A}}) \bm{v}_{m+1} \bm{e}_m^\ast.
\end{align*}
For the choice $\xi_{m-1}=\infty$, which is computationally attractive as it only requires one matrix-vector product in the last rational Krylov iteration, we have that $\bm{K}_m\in\C^{m \times m}$ is invertible \cite{beckermann2009error,guttel2013rational} and hence we obtain
\begin{equation}\label{eq:rat_krylov_relation}
\widetilde{\bm{A}}\bm{V}_m = \bm{V}_m \bm{H}_m \bm{K}_m^{-1} + h_{m+1,m} \bm{v}_{m+1} \bm{e}_m^\ast \bm{K}_m^{-1}.
\end{equation}

\begin{definition}[\cite{beckermann2009error}]
The rational Krylov relation with $\xi_{m-1}=\infty$ leads to the following rational matrix function approximation:
\begin{equation}\label{eq:fAb_rational_krylov}
f(\widetilde{\bm{A}})\widetilde{\bm{c}} \approx  \|\widetilde{\bm{c}}\|_2 \bm{V}_m f(\bm{H}_m\bm{K}_m^{-1})\bm{e}_1.
\end{equation}
\end{definition}

\begin{sloppypar}
The orthonormal basis $\bm{V}_m$ of $\mathcal{Q}_m(\widetilde{\bm{A}},\widetilde{\bm{c}})$ can be obtained by a slight modification of the polynomial Arnoldi method \cite{arnoldi1951principle,golub2013matrix}:
Ruhe's rational Arnoldi algorithm \cite{ruhe1984rational,ruhe1994rational,ruhe1994rational2,ruhe1994rational3,ruhe1998rational,berljafa2015generalized} replaces the matrix-vector product $\bm{x}_{j+1} = \widetilde{\bm{A}}\bm{v}_j$ in the $j$th iteration by the computation of a suitable continuation vector $\widetilde{\bm{v}}_j$ and \mbox{$\bm{x}_{j+1} = (\bm{I} - \widetilde{\bm{A}}/\xi_j)^{-1}\widetilde{\bm{A}}\widetilde{\bm{v}}_j$}, i.e., each rational Krylov iteration introduces one factor of the denominator polynomial.
The rest of the method, i.e., (modified) Gram--Schmidt orthogonalization against all previous basis vectors and normalization remains the same.
Computationally, one iteration of a rational Krylov subspace methods is significantly more expensive than one iteration of a polynomial Krylov methods due to the requirement to solve a linear system.
Our goal in the following two subsections is to construct a framework in which the superior approximation quality of rational functions can compensate for this additional cost in certain situations.
\end{sloppypar}

\subsection{Pole selection}\label{sec:rk_poles}

The choice of poles $\xi_j\in\mathbb{C}\cup\{\infty\}, j=1, \dots , m-1$ defines the space of rational functions representable by $\mathcal{Q}_m(\widetilde{\bm{A}},\widetilde{\bm{c}})$ and hence crucially determines the approximation quality of \cref{eq:fAb_rational_krylov}.
Rational (best) approximation results to the exponential function $e^{-x}$ on the real positive semi-axis date back several decades \cite{cody1969chebyshev,carpenter1984extended,gallopoulos1992efficient}.
As in our notation, we approximate $e^{h_i\widetilde{\bm{A}}}$ with negative semi-definite $\widetilde{\bm{A}}$, we consider the equivalent problem of approximating $e^x$ on the real negative semi-axis, which requires a change of signs of the poles $\xi_j$ obtained in the usual notation in the literature.
Hence, our first candidates of poles are the negative of the (complex conjugated) roots of the denominator polynomials of rational best approximations \cite{cody1969chebyshev,carpenter1984extended,gallopoulos1992efficient}.
Note that the real part of these poles are distributed over the positive and negative axis.

\begin{sloppypar}
An alternative method for optimal pole selection for arbitrary parameter-dependent functions was proposed in \cite{berljafa2015generalized,berljafa2017rkfit} and implemented in the \texttt{RKFIT} method \cite{berljafa2014rational}.
The method requires the specification of sample points within the spectrum of $\widetilde{\bm{A}}$ as well as a range of values for $h_i$.
It yields poles for general rational functions of type $(m+k,m)$, i.e., with numerator degree $m+k$ and denominator degree $m$.
Additionally, the poles' real part can be restricted to the negative complex half plane.
We use this option to obtain optimal poles for approximating $e^{-x}$ on the positive real semi-axis and subsequently take the negative of the poles to meet our notational requirement such that all $\xi_j$ have positive real parts.
\end{sloppypar}

Finally, we recap the idea that led to shift \& invert Krylov methods \cite{moret2004rd,van2006preconditioning}.
It has been shown that the restriction of poles to the real numbers leads to an optimal pole selection consisting of one repeated real pole \cite{borwein1983rational}.
Such optimal repeated real poles have been reported in \cite{borner2015three} in the similar setting of approximating the matrix exponential of a semi-definite matrix for a range of time step sizes.
In our notation, this approach leads to positive real poles $\xi_1 = \dots = \xi_{m-1}$.
The authors of \cite{borner2015three} additionally introduce cyclically repeated sets of two, three, and four real poles.
As these choices of poles did not noticeably improve our numerical results, we restrict our discussion in \Cref{sec:experiments} to the case of one single repeated pole.

In our numerical experiments, we choose default values of $72$ repeated real poles and $30$ complex conjugated poles for the rational best approximation and the \texttt{RKFIT} poles.

\subsection{Linear system solves}\label{sec:rk_linear_system_solves}

While matrix-vector products represent the computational bottleneck of polynomial Krylov methods, this is even more true for the linear system solves required by rational Krylov subspace methods.
The only way for rational Krylov subspace methods to outperform polynomial ones is by requiring much smaller iteration numbers such that the cost of the linear system solves is compensated by the avoidance of a large number of polynomial Krylov iterations.
Hence, the efficiency of the solution of the sequence of shifted linear systems $(\bm{I} - \widetilde{\bm{A}}/\xi_j)^{-1}\widetilde{\bm{A}}\widetilde{\bm{v}}_j$ as well as the ratio of required iteration numbers determines whether we can benefit from rational approximations in terms of runtime.
The optimization of the latter has been addressed in \Cref{sec:rk_poles};
we now turn to the efficiency of the linear system solves.

Defining $\bm{b}_j:=\widetilde{\bm{A}}\widetilde{\bm{v}}_j$ we rewrite the rational Arnoldi update as
\begin{equation*}
\bm{x}_{j+1} = (\bm{I} - \widetilde{\bm{A}}/\xi_j)^{-1}\bm{b}_j \Leftrightarrow (\bm{I} - \widetilde{\bm{A}}/\xi_j)\bm{x}_{j+1} = \bm{b}_j \Leftrightarrow (\xi_j \bm{I} - \widetilde{\bm{A}})\bm{x}_{j+1} = \xi_j\bm{b}_j.
\end{equation*}
Inserting the definition of $\widetilde{\bm{A}}$ from \Cref{thm:al-mohy_higham} and introducing subscripts indicating block sizes leads to the block linear system
\begin{equation}\label{eq:block_linear_system}
(\xi_j\bm{I}_{n+p} - \widetilde{\bm{A}})\bm{x}_{j+1} =
\begin{bmatrix}
\xi_j\bm{I}_n + \bm{A} & -\bm{C}\\
\bm{0} & \xi_j\bm{I}_p - \bm{J}_p
\end{bmatrix}
\begin{bmatrix}
[\bm{x}_{j+1}]_n\\
[\bm{x}_{j+1}]_p
\end{bmatrix}
=
\xi_j
\begin{bmatrix}
[\bm{b}_j]_n\\
[\bm{b}_j]_p
\end{bmatrix},
\end{equation}
where $p\ll n$.
The bottom set of equations can be solved for $[\bm{x}_{j+1}]_p$ efficiently as $(\xi_j\bm{I}_p - \bm{J}_p)$ is small and upper triangular.
Backsubstituting $[\bm{x}_{j+1}]_p$ into the top set of equations leads to the following shifted linear systems of equations:
\begin{equation}\label{eq:sequence_shifted_linear_systems}
(\xi_j\bm{I}_n + \bm{A}) [\bm{x}_{j+1}]_n = \xi_j [\bm{b}_j]_n + \bm{C} [\bm{x}_{j+1}]_p.
\end{equation}
Since $\bm{A}$ and the poles $\xi_j$ are constant across all time steps, each rational Krylov procedure requires solutions with the same linear system matrices but generally with different right-hand sides.
The difficulty of this problem is crucially affected by the choice of poles $\xi_j$:
we discussed in \Cref{sec:rk_poles} that the poles obtained from rational best approximations contain positive and negative real parts, which makes some systems \cref{eq:sequence_shifted_linear_systems} indefinite and complex-valued and hence more difficult to solve.
The repeated real pole as well as the real parts of the \texttt{RKFIT} poles are chosen positively to make all systems \cref{eq:sequence_shifted_linear_systems} strictly positive definite.

We now present two strategies for the efficient numerical solution of \cref{eq:sequence_shifted_linear_systems}.

The first strategy is to employ direct methods, which require the upfront computation of one LU or Cholesky decomposition of $(\xi_j\bm{I}_n + \bm{A})$ for each pole.
The obtained triangular matrices then allow for relatively cheap subsequent linear system solves by forward and backward substitution \cite{golub2013matrix}.
The direct approach is favorable for sufficiently small matrices and small numbers of different poles, i.e., few decompositions are required or if many time steps offer the opportunity to compensate the (potentially expensive) upfront computation of the decompositions.
Drawbacks of the direct approach are its generally cubic computational complexity as well as the fill-in issue \cite{golub2013matrix}.
The latter can be partially circumvented by row and column permutations.
In our numerical experiments, we rely on the software package Pardiso 6.0\footnote{\url{https://www.pardiso-project.org/}} \cite{petra2014real,petra2014augmented} for the direct solution of \cref{eq:sequence_shifted_linear_systems} as we observed a superior performance compared to Matlab's \texttt{amd} and \texttt{lu} functionality.

The second strategy is to employ iterative solvers \cite{saad2003iterative}, which do not suffer from the drawbacks discussed for direct solvers.
Unfortunately, prominent methods such as MINRES or GMRES \cite{saad2003iterative} are also based on polynomial Krylov subspaces and hence suffer from the very issue of increasing subspace sizes this work means to avoid.
Also restarted Krylov-based methods tailored to the solution of sequences of shifted linear systems were found to suffer from the described behavior \cite{simoncini2003restarted}.

A powerful technique capable of inducing convergence of iterative solvers independent of the problem size is preconditioning \cite{saad2003iterative}, which has already been employed in rational Krylov methods for certain matrix functions, cf.\ e.g., \cite{bertaccini2021computing}.
Since $\bm{A}$ is symmetric positive semi-definite, algebraic multigrid (AMG) methods \cite{ruge1987algebraic,falgout2006introduction} are well-suited for \cref{eq:sequence_shifted_linear_systems} when $\xi_j$ has positive real part\footnote{Note that for the very structured two-dimensional finite difference discretizations discussed in \Cref{def:finite_differences}, a geometric multigrid solver should also yield satisfactory results.}.
The general idea behind AMG is the construction of a hierarchy of linear systems of increasingly reduced size by means of smoothing and coarse-grid correction.
The solution of the reduced version of the original problem can be obtained cheaply and transformed back to the original problem setting.
We perform our numerical experiments with the aggregation-based multigrid package AGMG 3.3.5\footnote{\url{http://agmg.eu/}} \cite{notay2010aggregation,napov2012algebraic,notay2012aggregation}, which is capable of handling complex-valued nonsymmetric and moderately indefinite linear systems.
We directly use the flexible conjugate gradient (FCG) method \cite{notay2000flexible} implemented in AGMG to solve \cref{eq:sequence_shifted_linear_systems}.
We also experimented with preconditioners based on the approximation of the Schur complement \cite{pearson2012new,bergermann2023preconditioning} but found this to require more runtime due to a relatively high number of Krylov iterations.

\subsection{A-posteriori error estimate}\label{sec:rk_error}

As discussed in \Cref{sec:expint_implementation}, state-of-the-art exponential integration software builds on a-posteriori error estimates of polynomial Krylov approximations of the action of the matrix exponential on vectors.
In order to use rational Krylov methods in the same adaptive manner, we require an a-posteriori error estimate similar to \cite[Theorem 5.1]{saad1992analysis} for the polynomial case.
Although a-priori estimates \cite{guttel2013rational} as well as estimates over time intervals \cite{druskin2009solution} and for the shift \& invert case \cite{van2006preconditioning} exist in the literature, we require the following a-posteriori error at a single time point $h_i\in\R_{>0}$.

\begin{theorem}\label{prop:rat_krylov_error}
\begin{sloppypar}
Let $\xi_{m-1}=\infty$, which leads to the rational Krylov relation \cref{eq:rat_krylov_relation}.
Then, the approximation error of the rational Krylov approximation $\|\widetilde{\bm{c}}\|_2 \bm{V}_m e^{h_i \bm{H}_m\bm{K}_m^{-1}} \bm{e}_1$ to $e^{h_i \widetilde{\bm{A}}}\widetilde{\bm{c}}$ is given by
\end{sloppypar}
\begin{multline}\label{eq:rat_kryl_error}
e^{h_i \widetilde{\bm{A}}}\widetilde{\bm{c}} - \|\widetilde{\bm{c}}\|_2 \bm{V}_m e^{h_i \bm{H}_m\bm{K}_m^{-1}} \bm{e}_1\\ = h_i \|\widetilde{\bm{c}}\|_2 h_{m+1,m} \sum_{k=1}^\infty \bm{e}_m^\ast \bm{K}_m^{-1} \varphi_k(h_i \bm{H}_m\bm{K}_m^{-1}) \bm{e}_1 (h_i \widetilde{\bm{A}})^{k-1} \bm{v}_{m+1}.
\end{multline}
\end{theorem}
\begin{proof}
For $h_i = \|\widetilde{\bm{c}}\|_2=1$, the proof is essentially analogous to that of \cite[Theorem 5.1]{saad1992analysis} with the rational Krylov relation \cref{eq:rat_krylov_relation} in place of the polynomial Krylov relation.
We define the unit norm vector $\bm{\hat{c}} := \widetilde{\bm{c}}/\|\widetilde{\bm{c}}\|_2$.

By the recurrence relation of $\varphi$-functions, cf.\ \Cref{def:phi_functions}, we have
\begin{equation}\label{eq:recurrence_relation}
\widetilde{\bm{A}}\varphi_{k+1}(\widetilde{\bm{A}}) = \varphi_k(\widetilde{\bm{A}}) - \varphi_k(0)\bm{I},
\end{equation}
and we define the rational Krylov approximation error of $\varphi_{k}(\widetilde{\bm{A}})\bm{\hat{c}}$ as
\begin{equation}\label{eq:error_rat_krylov_approx_phij}
\bm{s}_m^{(k)} = \varphi_{k} (\widetilde{\bm{A}})\bm{\hat{c}} - \bm{V}_m \varphi_{k}(\bm{H}_m\bm{K}_m^{-1})\bm{e}_1.
\end{equation}
Then, we have for all $k\in\N_0$
\begin{align}
\varphi_k (\widetilde{\bm{A}})\bm{\hat{c}} & \overset{\cref{eq:recurrence_relation}}{=} \varphi_k(0)\bm{\hat{c}} + \widetilde{\bm{A}}\varphi_{k+1}(\widetilde{\bm{A}})\bm{\hat{c}}\notag\\
& \overset{\cref{eq:error_rat_krylov_approx_phij}}{=} \varphi_k(0)\bm{\hat{c}} + \widetilde{\bm{A}}\left( \bm{V}_m \varphi_{k+1}(\bm{H}_m\bm{K}_m^{-1})\bm{e}_1 + \bm{s}_m^{(k+1)} \right)\notag\\
& \overset{\cref{eq:rat_krylov_relation}}{=} \varphi_k(0) \bm{V}_m \bm{e}_1 + \bm{V}_m \underbrace{\bm{H}_m \bm{K}_m^{-1} \varphi_{k+1}(\bm{H}_m\bm{K}_m^{-1})}_{\overset{\cref{eq:recurrence_relation}}{=} \varphi_k(\bm{H}_m\bm{K}_m^{-1}) - \varphi_k(0)\bm{I}}\bm{e}_1\notag\\
& \qquad + h_{m+1,m}\bm{v}_{m+1} \bm{e}_m^\ast\bm{K}_m^{-1} \varphi_{k+1}(\bm{H}_m\bm{K}_m^{-1})\bm{e}_1 + \widetilde{\bm{A}} \bm{s}_m^{(k+1)}\notag\\
& = \bm{V}_m \varphi_k(\bm{H}_m\bm{K}_m^{-1})\bm{e}_1 + h_{m+1,m} \bm{e}_m^\ast \bm{K}_m^{-1}\varphi_{k+1}(\bm{H}_m\bm{K}_m^{-1})\bm{e}_1 \bm{v}_{m+1} + \widetilde{\bm{A}}\bm{s}_m^{(k+1)}.\label{eq:phi_k_error}
\end{align}
Inserting this into \cref{eq:error_rat_krylov_approx_phij} gives
\begin{equation}\label{eq:s1_s2_relation}
\bm{s}_m^{(k)} = h_{m+1,m} \bm{e}_m^\ast \bm{K}_m^{-1}\varphi_{k+1}(\bm{H}_m\bm{K}_m^{-1})\bm{e}_1 \bm{v}_{m+1} + \widetilde{\bm{A}}\bm{s}_m^{(k+1)}.
\end{equation}
Considering \cref{eq:phi_k_error} for $k=0$ and recursively inserting \cref{eq:s1_s2_relation} for $k=1, \dots , j-1$ yields
\begin{equation}\label{eq:rat_krylov_error_almost_there}
e^{\widetilde{\bm{A}}}\bm{\hat{c}} = \bm{V}_m e^{\bm{H}_m\bm{K}_m^{-1}}\bm{e}_1 + h_{m+1,m} \sum_{k=1}^j \bm{e}_m^\ast \bm{K}_m^{-1}\varphi_k(\bm{H}_m\bm{K}_m^{-1})\bm{e}_1 \widetilde{\bm{A}}^{k-1}\bm{v}_{m+1} + \widetilde{\bm{A}}^j\bm{s}_m^{(j)}.
\end{equation}
Letting $j\rightarrow\infty$ leads to the desired result for $h_i = \|\widetilde{\bm{c}}\|_2 = 1$.
As argued in the proof of \cite[Theorem 5.1]{saad1992analysis}, the error expansion convergence since we have $\widetilde{\bm{A}}^j\bm{s}_m^{(j)} \rightarrow 0$ as $\bm{s}_m^{(j)}\leq \frac{C}{j!}$ for a constant $C$.
The claim for general \mbox{$h_i\in\R_{>0}$} and \mbox{$\widetilde{\bm{c}}\in\C^n$} follows from \cref{eq:rat_krylov_error_almost_there} when replacing $\widetilde{\bm{A}}, \bm{H}_m,$ and $h_{m+1,m}$ by $h_i\widetilde{\bm{A}}, h_i\bm{H}_m,$ and $h_i h_{m+1,m}$, respectively, which is obtained when multiplying \cref{eq:rat_krylov_relation} by $h_i$, and inserting the definition of $\bm{\hat{c}}$.
\end{proof}

As the summands on the right hand side of \cref{eq:rat_kryl_error} typically decay rapidly \cite{saad1992analysis}, we obtain the following practical and cheaply computable a-posteriori error estimate.

\begin{corollary}\label{cor:rat_krylov_error}
\Cref{prop:rat_krylov_error} leads to the practical error estimate
\begin{equation}\label{eq:rk_error_estimate}
\| e^{h_i \widetilde{\bm{A}}}\widetilde{\bm{c}} - \|\widetilde{\bm{c}}\|_2 \bm{V}_m e^{h_i \bm{H}_m\bm{K}_m^{-1}} \bm{e}_1 \|_2 \approx h_i \|\widetilde{\bm{c}}\|_2 h_{m+1,m} \left| \bm{e}_m^\ast \bm{K}_m^{-1} \varphi_1(h_i \bm{H}_m\bm{K}_m^{-1}) \bm{e}_1 \right|.
\end{equation}
The error estimate can be computed by defining
$$
\bm{M}_{m+1} :=
\begin{bmatrix}
\bm{H}_m\bm{K}_m^{-1} & \bm{e}_1\\
\bm{0}^T & 0
\end{bmatrix}\in\C^{(m+1) \times (m+1)},
$$
which, by \cite[Theorem 1]{sidje1998expokit}, leads to
$$
e^{h_i\bm{M}_{m+1}} = \begin{bmatrix}
e^{h_i\bm{H}_m\bm{K}_m^{-1}} & h_i \varphi_1(h_i\bm{H}_m\bm{K}_m^{-1})\bm{e}_1\\
\bm{0}^T & 1
\end{bmatrix}.
$$
\begin{sloppypar}
We then define $\bm{w} = h_i \varphi_1 (h_i \bm{H}_m\bm{K}_m^{-1}) \bm{e}_1$, solve the (small) linear system \mbox{$\bm{K}_m \bm{u} = \frac{1}{h_i}\bm{w}$}, and obtain the error estimate as $h_i \|\widetilde{\bm{c}}\|_2 h_{m+1,m} | \bm{e}_m^\ast \bm{u} |$.
This only introduces a minimal extra cost as the computation of $e^{h_i \bm{H}_m\bm{K}_m^{-1}}$ is required for the approximation to $e^{h_i \widetilde{\bm{A}}}\widetilde{\bm{c}}$.
\end{sloppypar}
\end{corollary}

In the following, we consider the example from \cite[Example 3.5]{guttel2013rational}.
We illustrate the effectivity of the a-posteriori error estimate from \Cref{cor:rat_krylov_error} while it has been shown that existing a-priori error bounds for rational Krylov approximations \cite[Corollary 3.4]{guttel2013rational} need not be sharp.

\begin{figure}
	\centering
	\captionsetup[subfigure]{oneside,margin={1.4cm,0cm}}
	\subfloat[1D Laplacian]{
		\includegraphics[width=.48\textwidth]{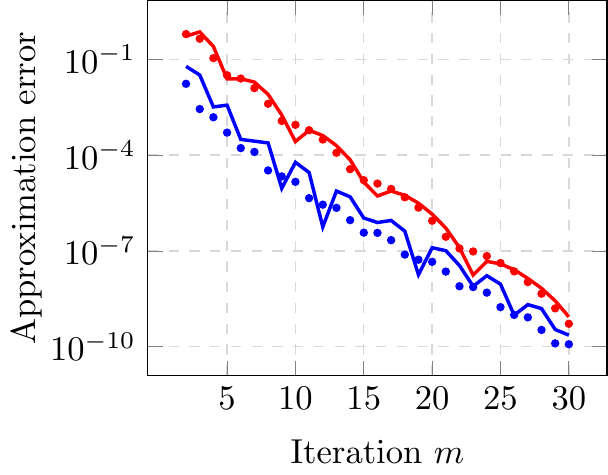}
	}
	\hfill
	\subfloat[2D Laplacian]{
		\includegraphics[width=.48\textwidth]{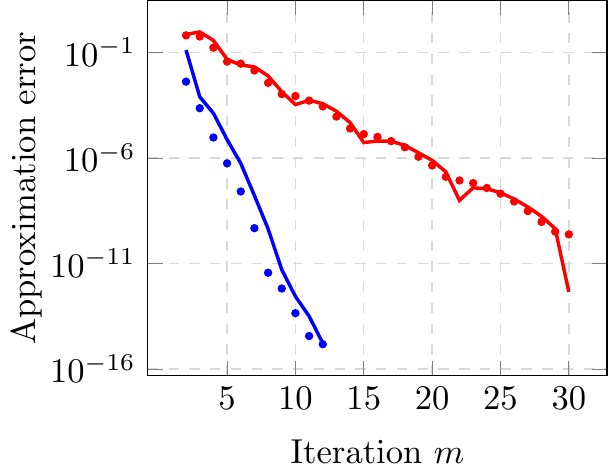}
	}
	
	\captionsetup[subfigure]{oneside,margin={-5.25cm,0cm}}
	\subfloat[Equispaced]{	
		\includegraphics[width=.97\textwidth]{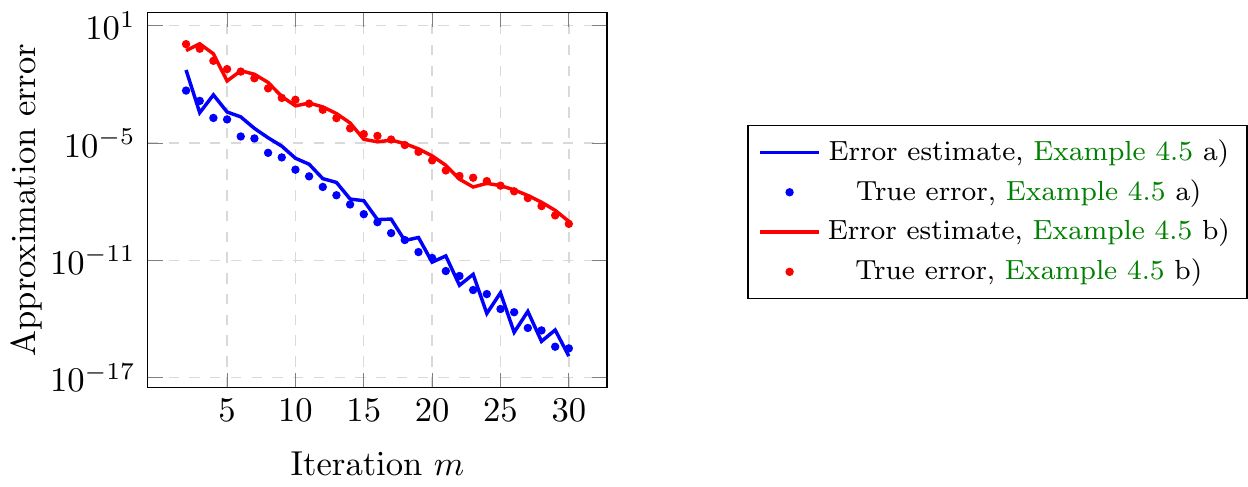}
	}
	\caption{Comparison of the a-posteriori error estimate from \Cref{cor:rat_krylov_error} with the explicitly computed $2$-norm errors of the rational Krylov approximation $e^{h_i \widetilde{\bm{A}}}\widetilde{\bm{c}} - \|\widetilde{\bm{c}}\|_2 \bm{V}_m e^{h_i \bm{H}_m\bm{K}_m^{-1}} \bm{e}_1$ for the two settings described in \Cref{ex:rat_kryl_error_estimate} and varying rational Krylov subspace dimension $m$.}\label{fig:rat_kryl_approx_error}
\end{figure}

\begin{example}\label{ex:rat_kryl_error_estimate}
Consider the three following test matrices $\widetilde{\bm{A}}_l\in\R^{900 \times 900}, l=1,2,3$, from \cite[Example 3.5]{guttel2013rational}, which are scaled and shifted to have equal spectra $\Sigma=[1, 1000]$:
the 1D Laplacian $\widetilde{\bm{A}}_1=\bm{T}_{900}$, the 2D Laplacian $\widetilde{\bm{A}}_2 = (\bm{T}_{30}\otimes\bm{I} + \bm{I}\otimes\bm{T}_{30})$, and the diagonal matrix $\widetilde{\bm{A}}_3=\text{diag}(1, \dots , 900)$ with $900$ evenly spaced eigenvalues.
\Cref{fig:rat_kryl_approx_error} compares the left and right hand sides of \cref{eq:rk_error_estimate} for these matrices in the following situations:
\begin{enumerate}[a)]
	\item $h_i=1$ and $\widetilde{\bm{c}}=\frac{1}{30}\bm{1}\in\R^{900}$, i.e., $\|\widetilde{\bm{c}}\|_2=1$,
	\item $h_i=0.01$ and $\widetilde{\bm{c}}\in\C^{900}$  with $\|\widetilde{\bm{c}}\|_2\approx 24.32$ where real and imaginary parts of the entries are drawn from uniform random distributions in $[0,1]$.
\end{enumerate}
\end{example}

\section{Algorithm}\label{sec:algorithm}

In this section, we summarize the ingredients introduced in the previous sections in \Cref{alg}.

\begin{algorithm}[t]
	\vspace{0.5em}
	\begin{tabular}{lll}
		Input:
		& $\bm{A}\in\R^{n \times n},$ & Discrete linear differential operator.\\
		& $g: [0,T] \times \R^n \rightarrow \R^n,$ & Semi-linear function.\\
		& $\bm{u_0}\in\R^n,$ & Initial conditions.\\
		& $[0, T]\subset\R_{\geq 0},$ & Time interval.\\
	\end{tabular}\\~\\~\\
	\begin{tabular}{ll}
		Parameters: & $h_i\in\R_{> 0}$; \texttt{tol}$\in\R_{> 0}$; \texttt{m\_min}, \texttt{m\_max} $\in\N$; $\xi_j\in\C, j = 1, \dots ,$\texttt{m\_max}\\
	\end{tabular}\\~\\
	\begin{tabular}{ll}
		Subroutines: & \texttt{exp\_rk\_int}, \texttt{exptAb\_routine}, \texttt{linear\_system\_solver}\\
	\end{tabular}\\
	
	\begin{algorithmic}[1]
		\If{\texttt{linear\_system\_solver} $== direct$}
		\State Compute decompositions of $(\xi_j\bm{I}_n + \bm{A})$ for $j=1, \dots ,$\verb|m_max|
		\EndIf
		\State \textbf{function} \texttt{exp\_rk\_int} \textit{\% solve \cref{eq:ode_system}}
		\For{every time step}
		\For{each linear combination of $\varphi$-functions}
		\State Assemble $\widetilde{\bm{A}}$ and $\widetilde{\bm{c}}$
		\State \textbf{function} \texttt{exptAb\_routine} \textit{\% approximate $e^{h_i\widetilde{\bm{A}}} \widetilde{\bm{c}}$}
		\While{\cref{eq:rk_error_estimate} $<$ \texttt{tol}}\label{line:alg:while}
		\State Compute continuation vector $\widetilde{\bm{v}}_j$ ($\widetilde{\bm{v}}_j=\bm{v}_j$ if not \texttt{rk2expint})
		\State Compute $\bm{b}_j = \widetilde{\bm{A}}\widetilde{\bm{v}}_j$
		\If{\texttt{exptAb\_routine} $==$ \texttt{rk2expint} \&\& $j<$ \texttt{m\_max}}
		\If{\texttt{linear\_system\_solver} $== direct$}
		\State Solve \cref{eq:block_linear_system} with back-subst.\ and the decomposition of $(\xi_j\bm{I}_n + \bm{A})$
		\ElsIf{\texttt{linear\_system\_solver} $== iterative$}
		\State Setup AGMG hierarchy for $(\xi_j\bm{I}_n + \bm{A})$
		\State Solve \cref{eq:block_linear_system} with back-subst.\ and iterative AGMG solver
		\EndIf
		\EndIf
		\State Extend Krylov decomposition, i.e, $\bm{V}_m$, $\bm{H}_m$ (and $\bm{K}_m$ if \texttt{rk2expint})
		\State Compute $\|\widetilde{\bm{c}}\|_2 \bm{V}_m e^{h_i \bm{H}_m\bm{K}_m^{-1}} \bm{e}_1$
		\EndWhile
		\State \textbf{end} \texttt{exptAb\_routine}
		\EndFor
		\State Update solution $\bm{u}$ for current time step according to \cref{eq:exprk1,eq:exprk2,eq:exprk3}
		\EndFor
		\State \textbf{end} \texttt{exp\_rk\_int}
	\end{algorithmic}
	\vspace{1em}
	\begin{tabular}{lll}
		Output: & $\bm{u}\in\R^{n \times n_t}$ & 
		\newdimen\origiwspc \origiwspc=\fontdimen2\font \fontdimen2\font=0.72ex Trajectory of the solution of \cref{eq:ode_system} along the $n_t$ time steps.\fontdimen2\font=\origiwspc
	\end{tabular}
	\caption{Rational Krylov Runge--Kutta exponential integrator method for the solution of \cref{eq:ode_system}.}\label{alg}
\end{algorithm}

\begin{sloppypar}
Our $(RK)^2$EXPINT (Rational Krylov Runge--Kutta exponential integrators, \texttt{rk2expint}) routine represents the core of the implementation of the method proposed in this paper that can be used as \texttt{exptAb\_routine} in place of \texttt{phipm} \cite{niesen2012algorithm} or \texttt{KIOPS} \cite{gaudreault2018kiops} presented in \Cref{sec:expint_implementation}.
\texttt{rk2expint} is based on \texttt{KIOPS} but replaces the polynomial Krylov method by the rational Krylov method introduced in \Cref{sec:rk} and implemented in the \texttt{RKToolbox} \cite{berljafa2014rational}.
We adopt the adaptivity from \texttt{KIOPS} with respect to the choice of the Krylov subspace size.
Note that due to this, the condition of the while-loop in line $9$ of \Cref{alg} as well as the quantity $\|\widetilde{\bm{c}}\|_2 \bm{V}_m e^{h_i \bm{H}_m\bm{K}_m^{-1}} \bm{e}_1$ in line $20$ are not evaluated in every iteration.
\texttt{KIOPS}' time interval sub-stepping functionality discussed in \Cref{sec:expint_implementation}, however, is excluded from \texttt{rk2expint} as the rational Krylov convergence should be independent of the spectrum of the discrete linear differential operator $\bm{A}$, cf.\ \Cref{sec:intro}.
Hence, the choice of a fixed number of optimized poles appropriate to the problem at hand should suffice.
In case of exhaustion of the a-priori specified poles, we continue extending the rational Krylov subspace by polynomial Krylov steps, i.e., poles $\xi_j=\infty$ within the while-loop.
This corresponds to restricting of the denominator degree in the rational approximation to the specified number of poles \texttt{m\_max} while further increasing the numerator degree.
\end{sloppypar}

We also adopt the functionality of \texttt{KIOPS} to perform the two tasks discussed at the end of \Cref{sec:expint_implementation}.
Note that the structure of higher-order exponential Runge--Kutta integrators increases the required number of calls of task 2 and additionally necessitates linear system solves with matrices such as $(\xi_j\bm{I}_n + h_i\bm{A})$ or $(\xi_j\bm{I}_n + \frac{h_i}{2}\bm{A})$.
Consequently, the benefit of a higher convergence order comes with the need of computing additional sets of matrix decompositions or AGMG hierarchies.
For readability and since only $\bm{A}$ would need to be changed into $h_i\bm{A}$ or $\frac{h_i}{2}\bm{A}$ at every appearance of $(\xi_j\bm{I}_n + \bm{A})$, we refrain from explicitly including this case in \Cref{alg}.

Finally, \texttt{rk2expint} relies on the a-posteriori error estimate derived in \Cref{sec:rk_error} as a stopping criterion to obtain approximations to $e^{h_i\widetilde{\bm{A}}} \widetilde{\bm{c}}$ to a user-specified tolerance \texttt{tol}, which we set to a default value of $10^{-8}$ in our numerical experiments.
The choice of poles $\xi_j$ and details on the solution of the linear systems \cref{eq:block_linear_system} are discussed in  \Cref{sec:rk_poles,sec:rk_linear_system_solves}, respectively.
We implement the example exponential Runge--Kutta integrators (\texttt{exp\_rk\_int} routines) \texttt{SW2} (Strehmel and Weiner \cite{weiner2013linear}), \texttt{ETD3RK} (Cox \& Mathews \cite{cox2002exponential}), and \texttt{Krogstad4} \cite{krogstad2005generalized} of stiff order 2, 3, and 4, respectively.
The default choice of the remaining parameters are \texttt{m\_min} $=5$ and \texttt{m\_max} $=72$ for one repeated real pole and \texttt{m\_max} $=30$ for complex poles for \texttt{rk2expint} as well as \texttt{m\_min} $=10$ and \texttt{m\_max} $=128$ if \texttt{phipm} or \texttt{KIOPS} is chosen as \texttt{exptAb\_routine}.
The default tolerance for the preconditioned linear system solves is $10^{-7}$.

\section{Numerical experiments}\label{sec:experiments}

\begin{sloppypar}
We test \Cref{alg} on finite difference discretizations of the Allen--Cahn and Gierer--Meinhardt equations defined on two-dimensional continuous domains as well as on inherently discrete graph/network domains.
All Matlab codes required to reproduce the results presented in this section are publicly available under \url{https://github.com/KBergermann/rk2expint}.
In our experiments, we used an AMD Ryzen 5 5600X 6-Core processor with $16$GB memory as well as Matlab R2020b with the external packages \texttt{phipm}\footnote{\url{http://www1.maths.leeds.ac.uk/~jitse/software.html}}, \texttt{KIOPS}\footnote{\url{https://gitlab.com/stephane.gaudreault/kiops}}, \texttt{RKToolbox}\footnote{\url{http://guettel.com/rktoolbox/}}, AGMG 3.3.5\footnote{\url{http://agmg.eu/}}, and Pardiso 6.0\footnote{\url{https://www.pardiso-project.org/}}.
\end{sloppypar}

The runtimes of the three methods \texttt{phipm}, \texttt{KIOPS}, and \texttt{rk2expint} are directly comparable since they are all implemented in Matlab and the \texttt{rk2expint} routine is based on \texttt{KIOPS}, which, in turn, is based on \texttt{phipm}.
Furthermore, the computational bottleneck of \texttt{rk2expint} is the solution of the sequences of shifted linear systems, which is performed by external software and makes up between $60\%$ and $90\%$ of the total runtime.

We mention that the techniques presented in \Cref{sec:rk_linear_system_solves} are also applicable to the linear system solves with the Jacobian within Newton iterations that one encounters when employing implicit (non-exponential) Runge--Kutta methods.
A class of suitable methods for \cref{eq:ode_system} are stiffly accurate diagonally implicit Runge--Kutta (SDIRK) methods \cite{alexander1977diagonally} for which two integrators \texttt{SDIRK}(2,2) of order $2$ with $2$ stages as well as one integrator \texttt{SDIRK}(3,3) of order $3$ with $3$ stages exist \cite[Theorem 5]{alexander1977diagonally}.
As for our method, the runtime of SDIRK methods is dominated by the solution of linear systems similar to \cref{eq:sequence_shifted_linear_systems} and using the preconditioned iterative strategy from \Cref{sec:rk_linear_system_solves}, their runtime depends on the required number of Newton iterations per time step.
Numerical experiments not detailed in this paper show that the latter tends to increase in comparison to the required number of rational Krylov iterations as the problem becomes ``more challenging'', i.e., when smaller time steps and larger denominator polynomial degrees are required to obtain stable solutions.
Roughly speaking, using similar tolerances, \texttt{SDIRK}(2,2) was about a factor of $5$ faster than \texttt{SW2} with \texttt{rk2expint} in the setting of \Cref{fig:iter_runtime_2D_AC}, \texttt{SDIRK}(3,3) was about a factor of $2$ faster than \texttt{ETD3RK} with \texttt{rk2expint} in the setting of \Cref{fig:iter_runtime_2D_GM}, and \texttt{SDIRK}(3,3) was somewhat slower than \texttt{Krogstad4} with \texttt{rk2expint} in the setting of \Cref{fig:iter_runtime_graph_AC}.
We chose \texttt{SDIRK}(3,3) in the latter example due to the lack of existence of an SDIRK method of order $4$ with $4$ stages \cite[Theorem 6]{alexander1977diagonally}.
Since such a method would be expected to be as accurate as \texttt{Krogstad4} at about $\frac{4}{3}$ of the runtime of \texttt{SDIRK}(3,3), we conclude that exponential integration is the superior strategy for this problem.

\subsection{Allen--Cahn equation on 2D continuous domain}\label{sec:experiments_2D_AC}

We start by considering the Allen--Cahn equation, which can be used to model phase separation phenomena without mass conservation\footnote{i.e., the integral over $u$ on the domain $\Omega$ may change over time.} \cite{allen1979microscopic}.
We adopt the example setting from \cite{gaudreault2018kiops} and define it as
\begin{equation}\label{eq:allen_cahn}
 \frac{\partial u}{\partial t} = \epsilon^2 \Delta u + u - u^3,
\end{equation}
with the interface parameter $\epsilon\in\R$, homogeneous Neumann boundary conditions, $\epsilon^2=0.1$, $\Omega=[-1,1]^2$, $T=1$, and initial conditions \mbox{$u_0 = 0.1 + 0.1 \cos(2\pi x) \cos(2\pi y)$}, where $x$ and $y$ denote the two spatial coordinates.

\Cref{fig:iter_runtime_2D_AC} in \Cref{sec:intro} compares our method \texttt{rk2expint} with \texttt{phipm} and \texttt{KIOPS} in approximating the quantities $e^{h_i\widetilde{\bm{A}}}\widetilde{\bm{c}}$ for a relatively large time step size $h_i=\frac{1}{2}$ in terms of average Krylov iteration numbers per time step and total runtimes for the solution of \cref{eq:allen_cahn}.
It confirms that rational Krylov iteration numbers are almost independent of $\|h_i\widetilde{\bm{A}}\|_2$ (i.e., the problem size, cf.\ \Cref{prop:fd_spectra}) leading to a near-linear scaling of the runtime while the polynomial Krylov iteration numbers of \texttt{phipm} and \texttt{KIOPS} increase with growing $n$.
Note that the structure of exponential Runge--Kutta methods does not permit \texttt{KIOPS} to outperform \texttt{phipm} as reported, e.g., in \cite{gaudreault2018kiops}.
The reason is the cost effectivity of EPIRK \cite{tokman2006efficient,tokman2011new} methods in terms of the numbers of quantities $e^{h_i \widetilde{\bm{A}}}\widetilde{\bm{c}}$ required per time step to obtain a given convergence order.
Combining our approach with EPIRK methods would be an interesting road for future research.

\begin{figure}
	\includegraphics[width=.65\textwidth]{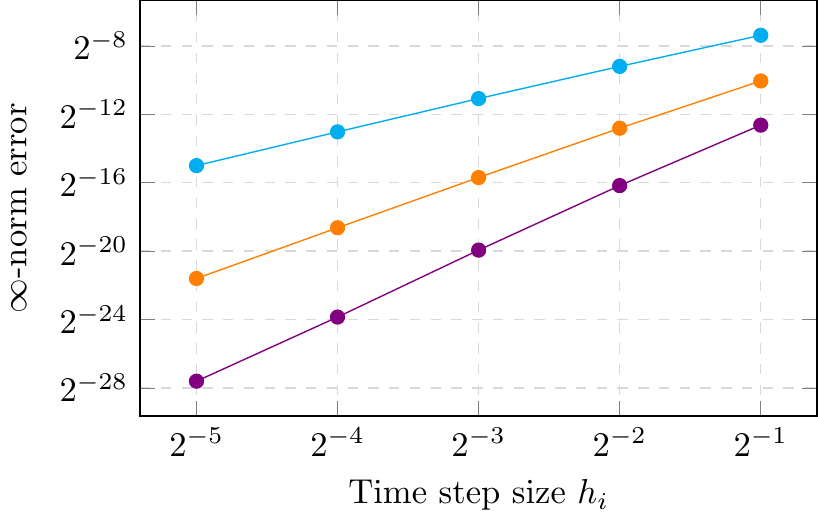}
	\hfill
	\raisebox{1\height}{
		\includegraphics[width=.2\textwidth]{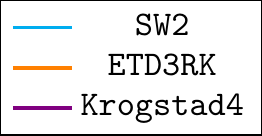}
	}
	\caption{Experimental convergence order of the exponential Runge--Kutta integrators \texttt{SW2}, \texttt{ETD3RK}, and \texttt{Krogstad4}, which have stiff order $2$, $3$, and $4$, respectively.
	We plot the error to Matlab's \texttt{ode15s} solution to a tolerance of $10^{-12}$ in $\infty$-norm.
	The problem setting is the 2D Allen--Cahn equation from \Cref{fig:iter_runtime_2D_AC} with $n_x=200$ and varying time step size $h_i$.
	We only plot the errors obtained by \texttt{rk2expint} as they are very similar for all three \texttt{exptAb\_routine}s.}\label{fig:AC_eoc}
\end{figure}

\begin{figure}
	\centering
	\subfloat[Krylov iteration numbers]{
		\includegraphics[width=.45\textwidth]{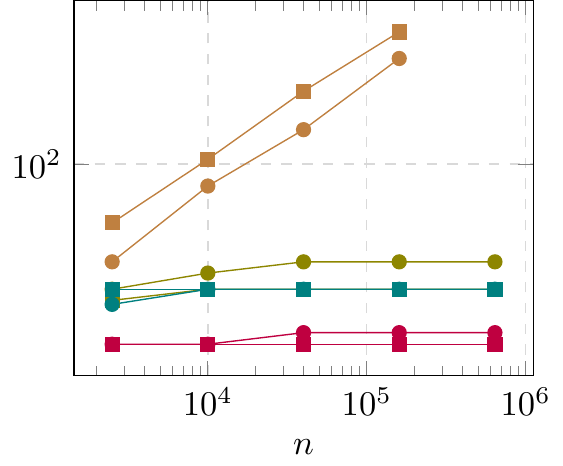}
	}
	\hfill
	\subfloat[Runtime in seconds]{
		\includegraphics[width=.45\textwidth]{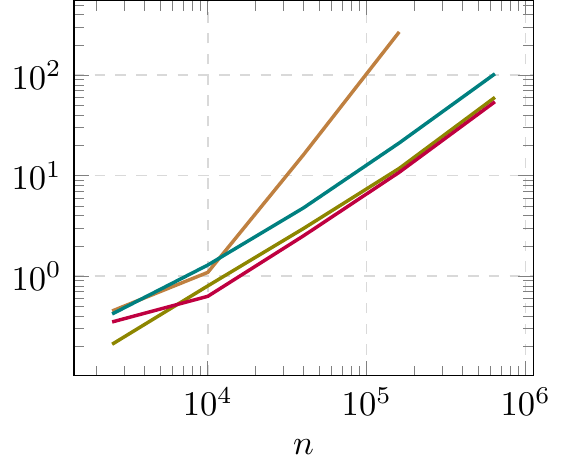}
	}
	\vspace{10pt}
	\includegraphics[width=.9\textwidth]{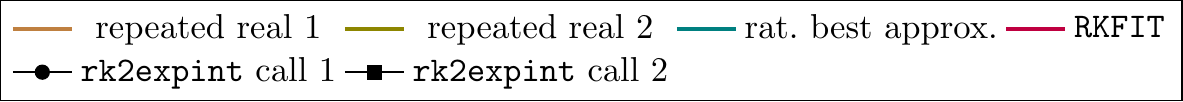}
	\vspace{-10pt}
	\caption{Rational Krylov iteration and runtime comparison for different choices of poles $\xi_j$ for the 2D Allen--Cahn equation solved with the \texttt{SW2} integrator.
		The problem settings corresponds to that of \Cref{fig:iter_runtime_2D_AC}.}\label{fig:2D_AC_comparison_poles}
\end{figure}

In \Cref{fig:AC_eoc}, we experimentally confirm the theoretically indicated convergence orders of the three considered exponential Runge--Kutta integrators \texttt{SW2}, \texttt{ETD3RK}, and \texttt{Krogstad4} in the example setting of \Cref{fig:iter_runtime_2D_AC} and for $n_x=200$.

In addition, \Cref{fig:2D_AC_comparison_poles} compares average Krylov iteration numbers per time step and total runtimes for the different choices of poles presented in \Cref{sec:rk_poles}.
We use two choices of one repeated real pole (corresponding to the special case of a shift \& invert Krylov subspace method \cite{moret2004rd,van2006preconditioning}), for which we have increased the maximum number of poles to $\texttt{m\_max}=500$ in this example in order to prevent polynomial Krylov steps after exhaustion of the provided poles.
Choice $1$ corresponds to $\xi_1 = \dots = \xi_{m-1}=-3.14 \cdot 10^{5}$ \cite{borner2015three} and choice $2$ to $\xi_1 = \dots = \xi_{m-1}=-\frac{h_i}{10}=-\frac{1}{20}$ \cite{van2006preconditioning}.
\Cref{fig:2D_AC_comparison_poles} shows that the choice of the repeated real pole has a significant influence on the convergence behavior and optimal pole selection strategies are a topic of ongoing research \cite{druskin2011adaptive,guttel2013rational,borner2015three,berljafa2017rkfit,massei2021rational}.
The two sets of complex-valued poles (rat.\ best approx.\ and \texttt{RKFIT}) both show low and virtually identical iteration numbers across all considered problem sizes with the \texttt{RKFIT} numbers ranging below those of the rational best approximations'.
We repeat the same experiment for the Gierer--Meinhardt equations in \Cref{fig:2D_GM_comparison_poles} and use \texttt{RKFIT} poles in the remainder of the numerical experiments.
We remark again that \texttt{RKFIT} allows automated pole optimization tailored for a wide range of problems, cf.~\Cref{sec:rk_poles}.

\subsection{Gierer--Meinhardt equations on 2D continuous domain}\label{sec:experiments_2D_GM}

\begin{figure}
	\includegraphics[width=0.135\textwidth]{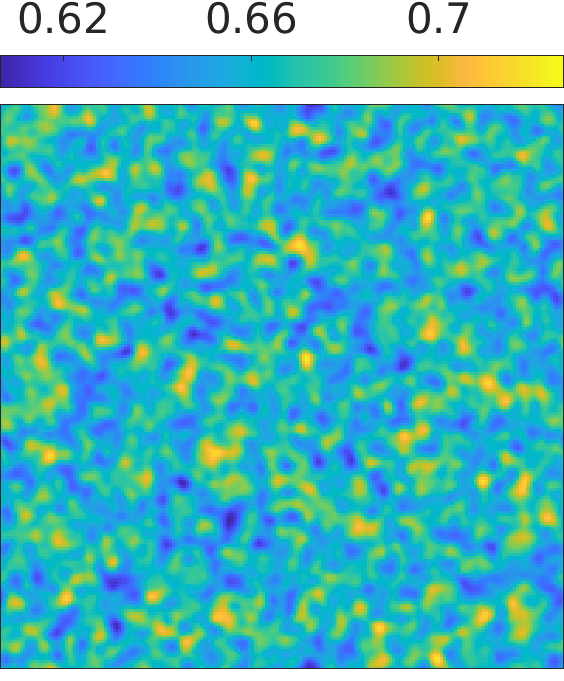}
	\includegraphics[width=0.135\textwidth]{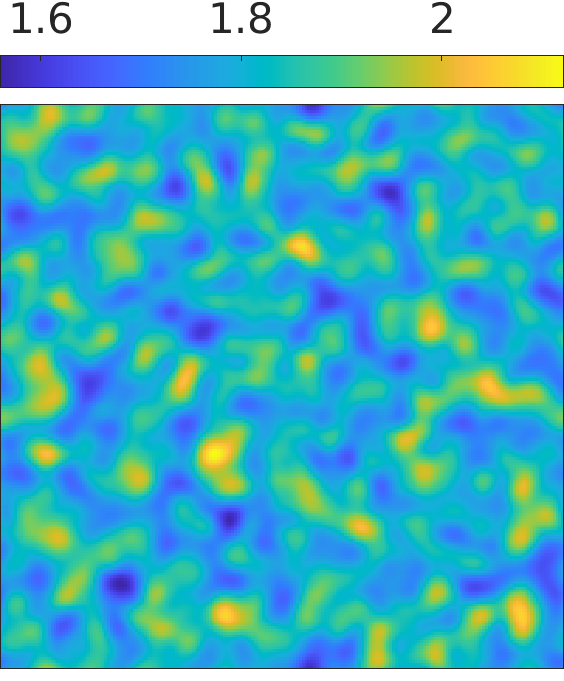}
	\includegraphics[width=0.135\textwidth]{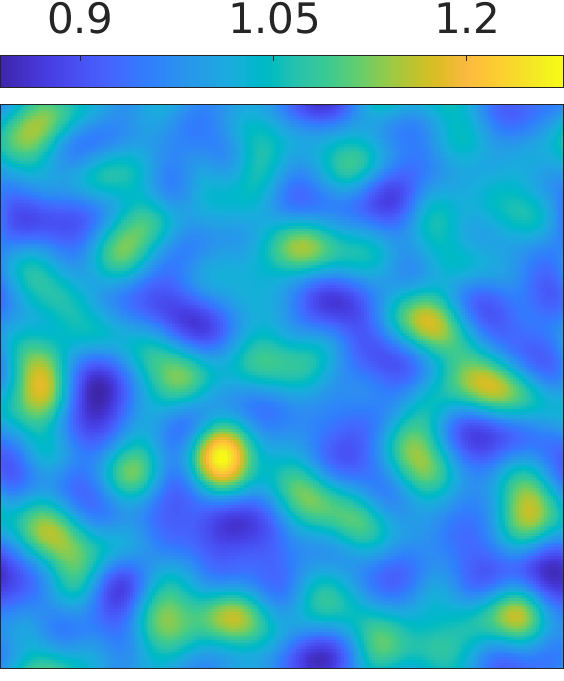}
	\includegraphics[width=0.135\textwidth]{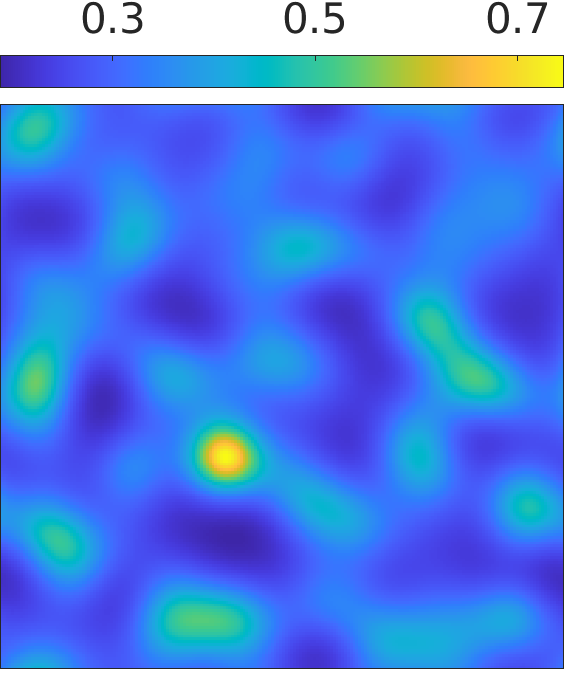}
	\includegraphics[width=0.135\textwidth]{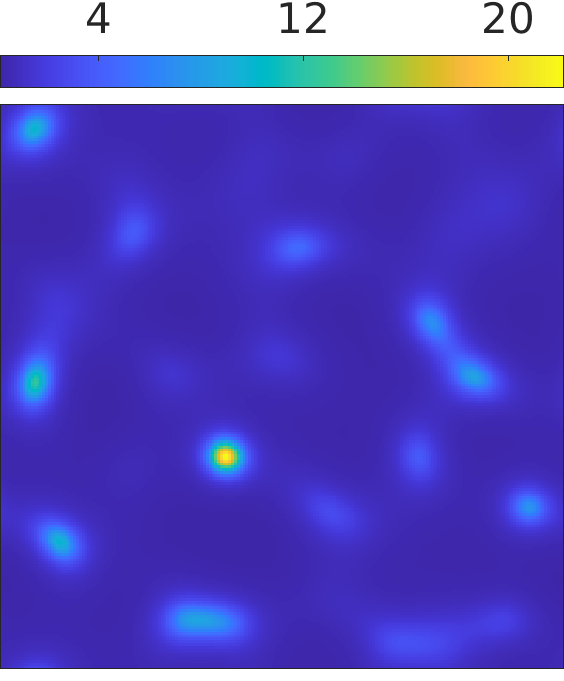}
	\includegraphics[width=0.135\textwidth]{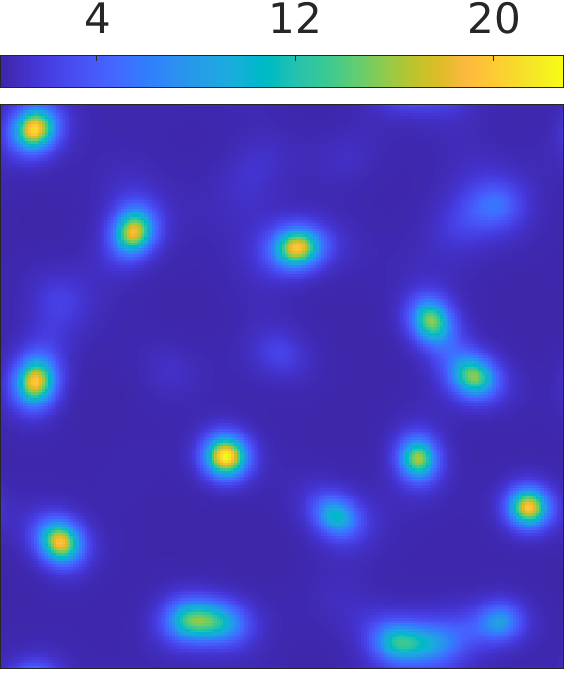}
	\includegraphics[width=0.135\textwidth]{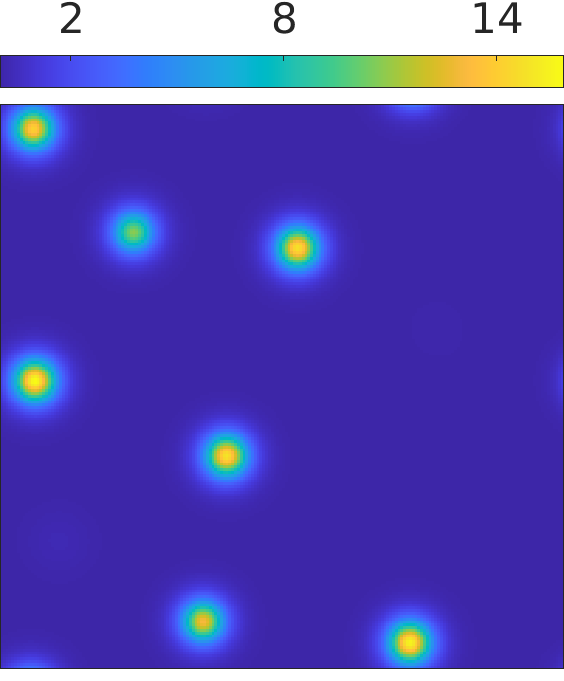}
	
	\vspace{2pt}
	
	\includegraphics[width=0.135\textwidth]{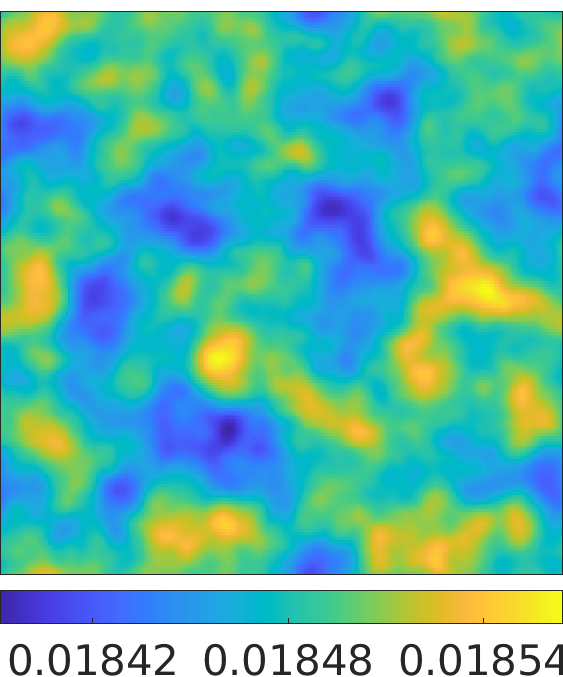}
	\includegraphics[width=0.135\textwidth]{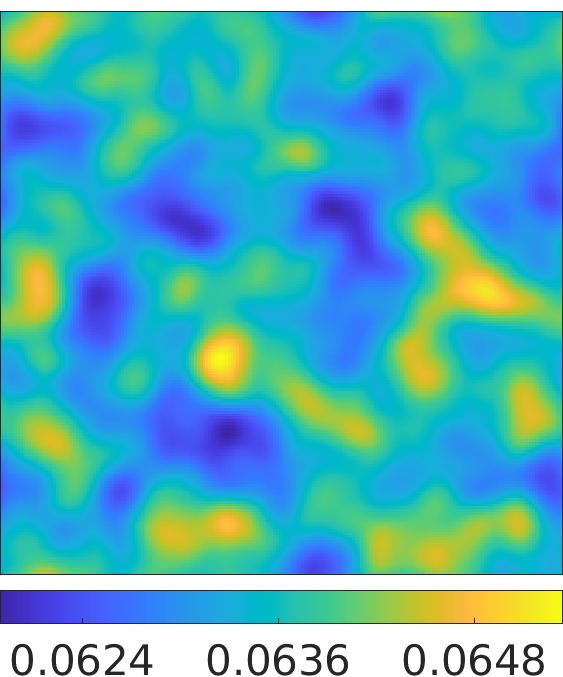}
	\includegraphics[width=0.135\textwidth]{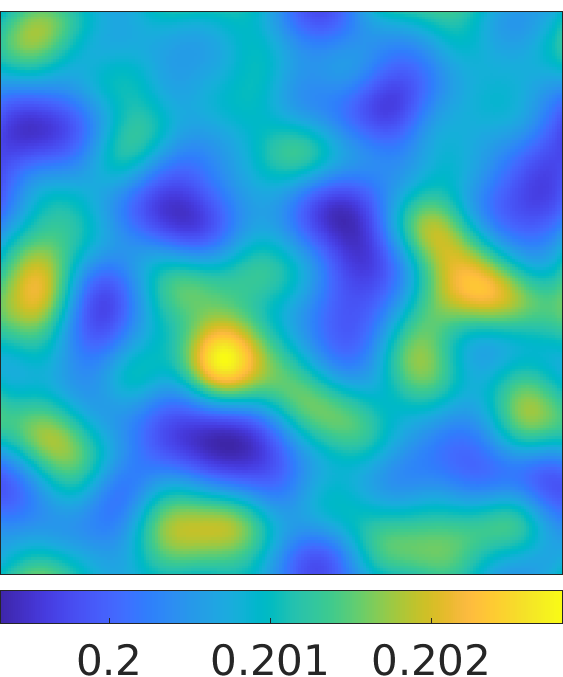}
	\includegraphics[width=0.135\textwidth]{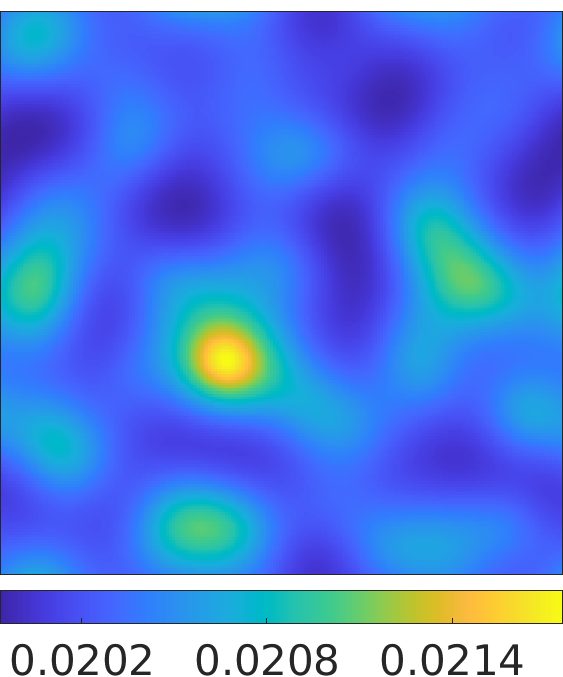}
	\includegraphics[width=0.135\textwidth]{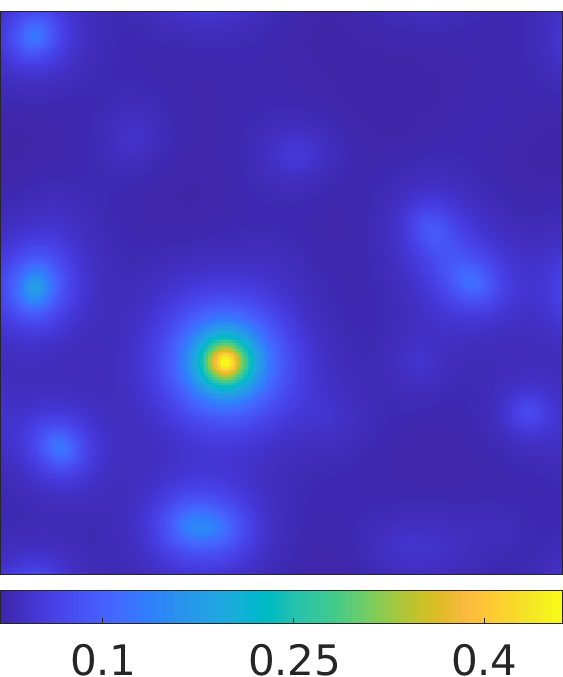}
	\includegraphics[width=0.135\textwidth]{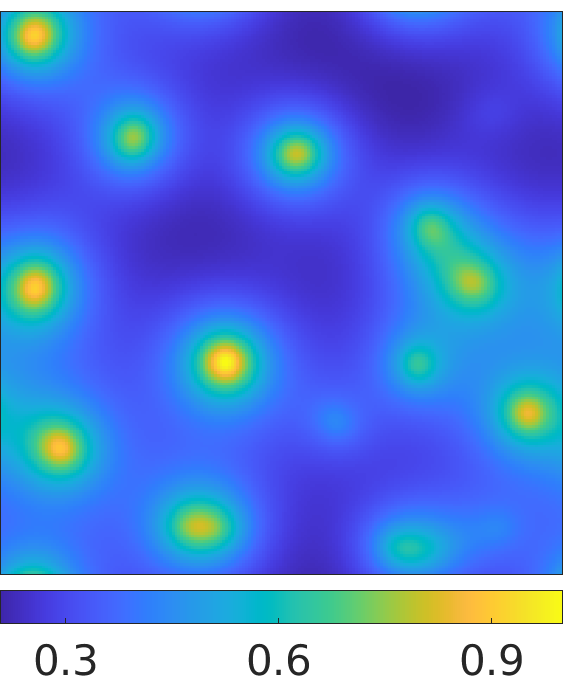}
	\includegraphics[width=0.135\textwidth]{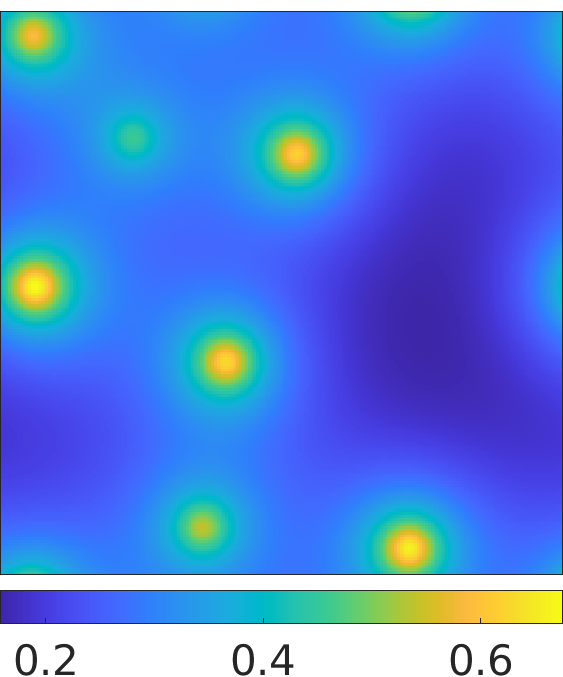}
	\caption{Gierer--Meinhardt simulation result on the 2D unit square with $n_x=200$.
	The top row shows activator and the bottom row inhibitor concentrations.
	Initial activator concentrations are random while initial inhibitor concentrations are constant.
	The parameters are chosen $D_a=0.005, D_h=0.5, p=\mu=p'=\nu = 16$.
	The time interval is $[0,1]$ with a time step size of $h_i=0.01$.
	From left to right, the corresponding times are $t=0.02, t=0.05, t=0.2, t=0.4, t=0.45, t=0.5, t=1$.}\label{fig:GM_example_solution}
\end{figure}

Next, we consider the Gierer--Meinhardt equations, which are frequently used to model biological pattern formation processes \cite{gierer1972theory}.
They describe the spatio-temporal evolution of an activator $a$ and an inhibitor $h$ and are given by
\begin{align}
\frac{\partial a}{\partial t} & = D_a \Delta a + p \frac{a^2}{h} - \mu a,\label{eq:pde_gm_a}\\
\frac{\partial h}{\partial t} & = D_h \Delta h + p' a^2 - \nu h,\label{eq:pde_gm_h}
\end{align}
where $D_a, D_h \in\R_{>0}$ denote the diffusion constants of activator and inhibitor, respectively, and $p, \mu, p', \nu\in\R_{>0}$ denote model parameters.
The two equations lead to block-diagonal discrete linear differential operators $\bm{A}\in\R^{2n_x^2 \times 2n_x^2}$ and block solution vectors $\bm{u}\in\R^{2n_x^2}$.
Throughout our experiments, we use periodic boundary conditions as well as random initial conditions in the interval $[0.4, 0.6]$ for $a$ and constant initial conditions of $0.2$ for $h$.
\Cref{fig:GM_example_solution} shows an exemplary trajectory of a solution of the Gierer--Meinhardt equations, where the activator concentration $a$ is shown in the top and the inhibitor concentration $h$ in the bottom row.

\begin{figure}
	\centering
	\subfloat[Krylov iteration numbers]{
		\includegraphics[width=.42\textwidth]{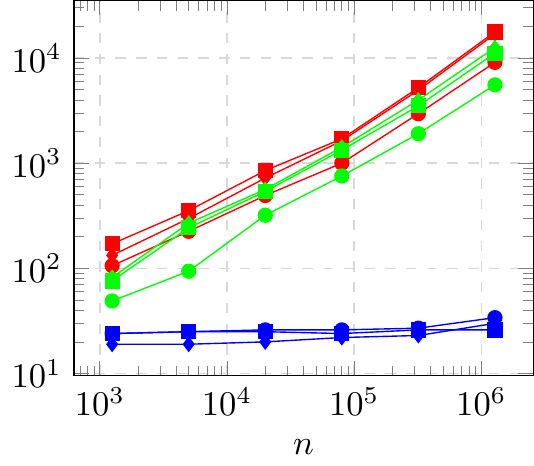}
	}
	\hfill
	\subfloat[Runtime in seconds]{
		\includegraphics[width=.42\textwidth]{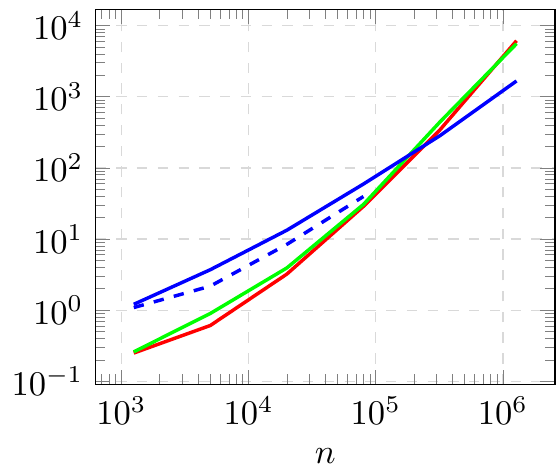}
	}
	\vspace{10pt}
	\includegraphics[width=.99\textwidth]{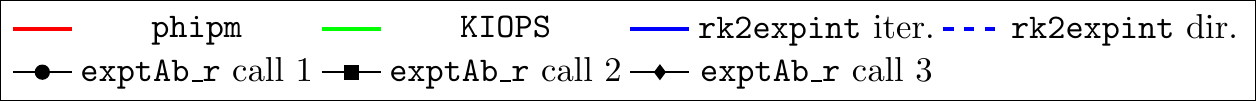}
	\vspace{0pt}
	\caption{Average Krylov iteration numbers per evaluation of $e^{h_i \widetilde{\bm{A}}}\widetilde{\bm{c}}$ and total runtimes of solving the 2D Gierer--Meinhardt equations with periodic boundary conditions.
	The total number of grid points is denoted by $n$.
	We use the \texttt{ETD3RK} exponential Runge--Kutta integrator, which requires three evaluations of $e^{h_i \widetilde{\bm{A}}}\widetilde{\bm{c}}$ per time step.
	For \texttt{rk2expint}, we use complex-valued $(35,30)$ \texttt{RKFIT} poles fitted on the interval $[0,10^6]$ and report runtimes of direct and preconditioned iterative linear system solves.}\label{fig:iter_runtime_2D_GM}
\end{figure}

\begin{figure}
	\centering
	\includegraphics[width=.48\textwidth]{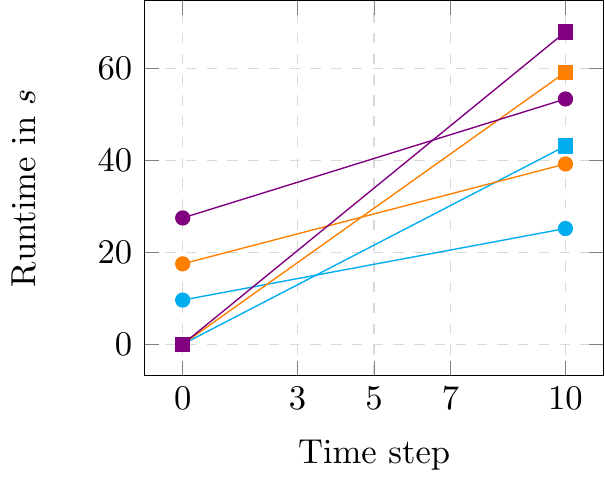}
	\hfill
	\raisebox{1\height}{
		\includegraphics[width=.28\textwidth]{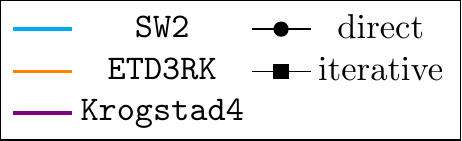}
	}
	\caption{Runtime comparison of direct and preconditioned iterative linear system solves within rational Krylov methods for different exponential Runge--Kutta integrators.
		The problem setting corresponds to that of \Cref{fig:iter_runtime_2D_GM} with $n_x=200$, i.e., $n=80\,000$.}\label{fig:GM_direct_vs_iterative}
\end{figure}

\Cref{fig:iter_runtime_2D_GM} compares Krylov iteration numbers and runtimes of \texttt{phipm}, \texttt{KIOPS}, and \texttt{rk2expint} for the Gierer--Meinhardt equations.
In these and the following experiments, we use the parameters \mbox{$D_a=0.01, D_h=p=p'=\mu=\nu=T=1$}, and $h_i=0.1$ on the unit square $\Omega=[0,1]^2$.
\Cref{fig:iter_runtime_2D_GM} confirms the observations made for the Allen--Cahn equation in \Cref{fig:iter_runtime_2D_AC}, namely approximately constant rational Krylov iteration numbers and a near-linear runtime dependence of \texttt{rk2expint} on the problem size.

Furthermore, we compare the performance of direct and preconditioned iterative linear system solves discussed in \Cref{sec:rk_linear_system_solves}.
\Cref{fig:GM_direct_vs_iterative} compares runtimes of the direct and preconditioned iterative solvers in the previously considered Gierer--Meinhardt problem setting for $n_x=200$ and for our three different exponential Runge--Kutta integrators.
It illustrates that the upfront cost of computing the decompositions for the direct solver quickly pays off in comparison to the runtime required by the preconditioned iterative solver.
Depending on how many decompositions per pole are required by the exponential integrator (cf.\ the discussion in \Cref{sec:algorithm}), the upfront cost is redeemed within $3$ to $7$ time steps making direct solvers particularly well-suited if many time steps are required.

The major limitation of the direct approach is its memory requirement: for \mbox{$n_x\geq 400$}, the required decompositions can no longer be stored in our $16$GB memory.
Possible remedies are using low-order exponential Runge--Kutta integrators or a smaller number of distinct (and possibly repeated) poles, both of which in turn lead to the requirement of performing either more time steps or more rational Krylov iterations to obtain the same accuracy of the solution to \cref{eq:ode_system}.
The slightly lower runtimes of the direct solver for admissible problem sizes up to $n_x=200$ are also reported in \Cref{fig:iter_runtime_2D_GM}.

Finally, \Cref{fig:2D_GM_comparison_poles} repeats the comparison of different choices of poles described at the end of \Cref{sec:experiments_2D_AC} for the Gierer--Meinhardt equations.
While the results are qualitatively similar to those reported in \Cref{fig:2D_AC_comparison_poles} on a generally higher level of iteration numbers, the runtime advantage of \texttt{RKFIT} in comparison to the other choices of poles are more pronounced for the (more challenging) Gierer--Meinhardt equations.
Note that iteration numbers of the rational best approximation poles include polynomial Krylov iterations performed after exhaustion of the $30$ available poles, cf.~\Cref{sec:rk_poles}.
\Cref{fig:2D_GM_comparison_poles} underpins the observation described in the beginning of \Cref{sec:experiments} that our method improves comparable methods on challenging problems.

\begin{figure}
	\centering
	\subfloat[Krylov iteration numbers]{
		\includegraphics[width=.43\textwidth]{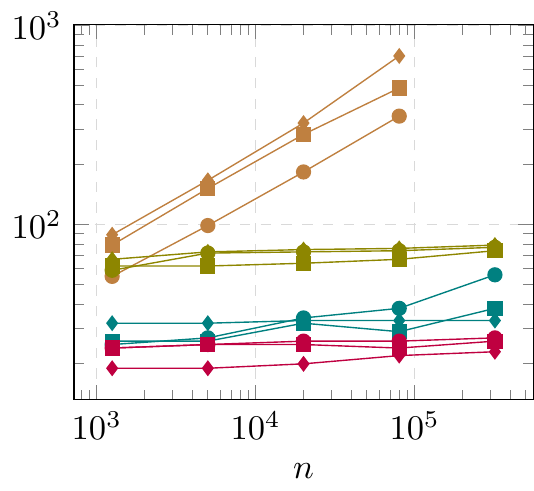}
	}
	\hfill
	\subfloat[Runtime in seconds]{
		\includegraphics[width=.43\textwidth]{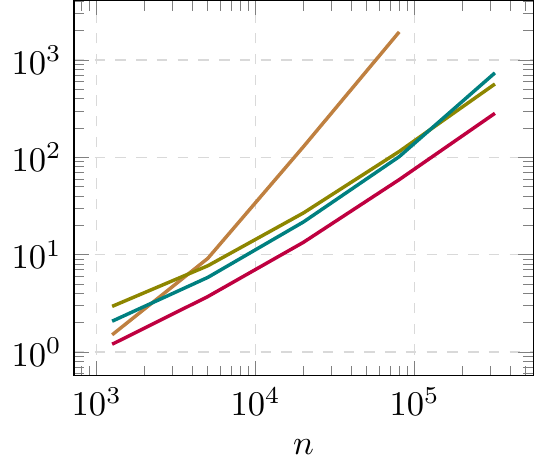}
	}
	\vspace{10pt}
	\includegraphics[width=.9\textwidth]{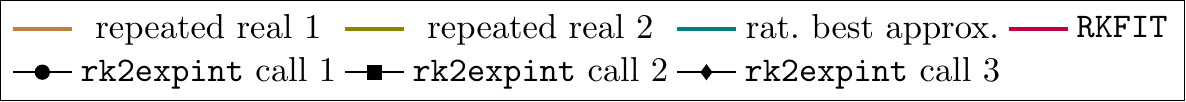}
	\vspace{-10pt}
	\caption{Rational Krylov iteration and runtime comparison for different choices of poles $\xi_j$ for the 2D Gierer--Meinhardt equation solved with the \texttt{ETD3RK} integrator.
	The problem settings corresponds to that of \Cref{fig:iter_runtime_2D_GM}.}\label{fig:2D_GM_comparison_poles}
\end{figure}

\subsection{Allen--Cahn equation on networks}

\begin{table}
	\centering
	\begin{tabular}{|l|rrrcc|}
		\hline\hline
		Network & $n$ & $|E|$ & $\lmax$ & $D$ & $\Sigma$\\\hline\hline
		minnesota & $2\,640$ & $6\,604$ & $6.88$ & $5 \cdot 10^3$ & $[0, 1.72 \cdot 10^3]$\\
		usroads (subset) & $17\,502$ & $46\,418$ & $8.22$ & $5 \cdot 10^4$ & $[0, 2.06 \cdot 10^4]$\\
		ak2010 & $42\,381$ & $204\,182$ & $3.32 \cdot 10^8$ & $10^{-2}$ & $[0, 1.66 \cdot 10^5]$\\
		luxembourg-osm & $114\,599$ & $239\,332$ & $7.21$ & $5 \cdot 10^6$ & $[0, 1.81 \cdot 10^6]$\\
		ny2010 & $350\,167$ & $1\,709\,544$ & $2.30 \cdot 10^7$ & $5$ & $[0, 5.75 \cdot 10^6]$\\
		roadNet-PA & $1\,087\,562$ & $3\,083\,028$ & $10.35$ & $5 \cdot 10^7$ & $[0, 2.59 \cdot 10^7]$\\\hline\hline
	\end{tabular}
	\caption{Example networks for the Allen--Cahn equation.
		The parameter $n$ denotes the number of nodes, $|E|$ the number of edges of the network, and $\lmax$ the largest eigenvalue of its unnormalized graph Laplacian.
		Furthermore, $D$ denotes the diffusion constant and $\Sigma$ the resulting spectrum of $\epsilon D \bm{A}$.
		For the usroads subset network, we filtered the longitudinal coordinates of the full usroads network for the interval $[-125, -115]$ in order to reduce the network size.}\label{tab:networks_AC}
\end{table}

\begin{figure}
	\includegraphics[width=0.16\textwidth]{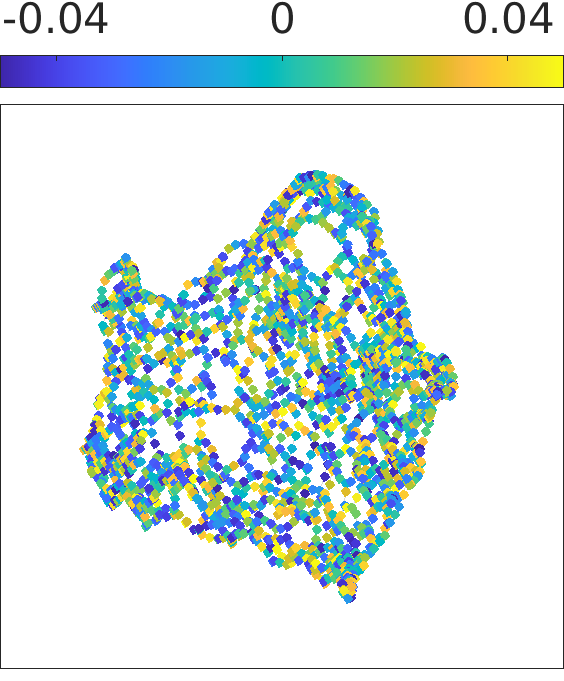}
	\includegraphics[width=0.16\textwidth]{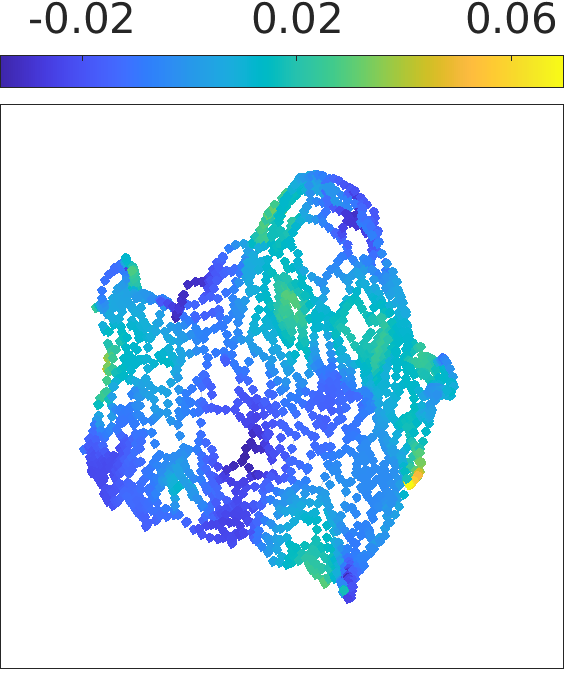}
	\includegraphics[width=0.16\textwidth]{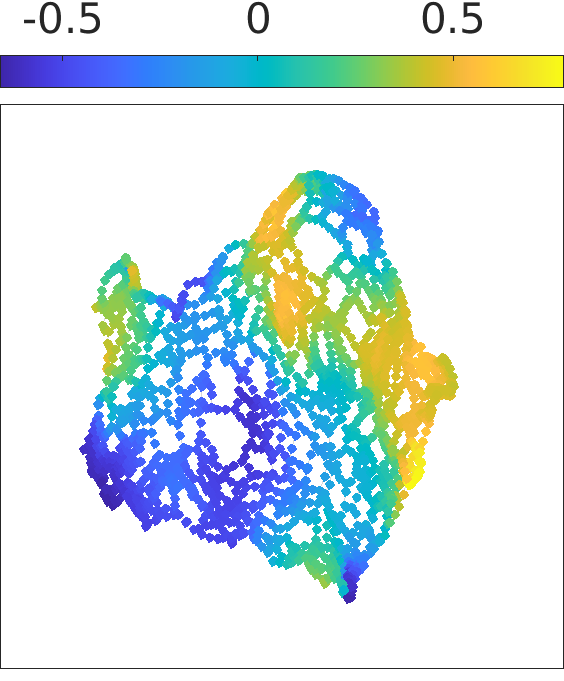}
	\includegraphics[width=0.16\textwidth]{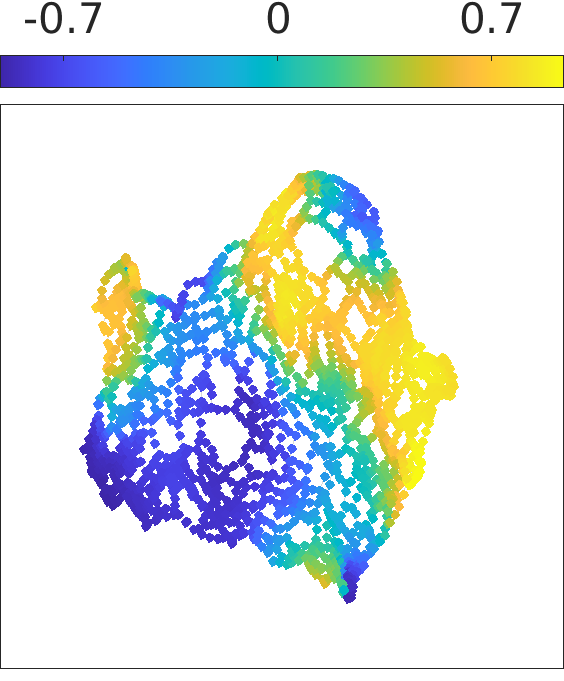}
	\includegraphics[width=0.16\textwidth]{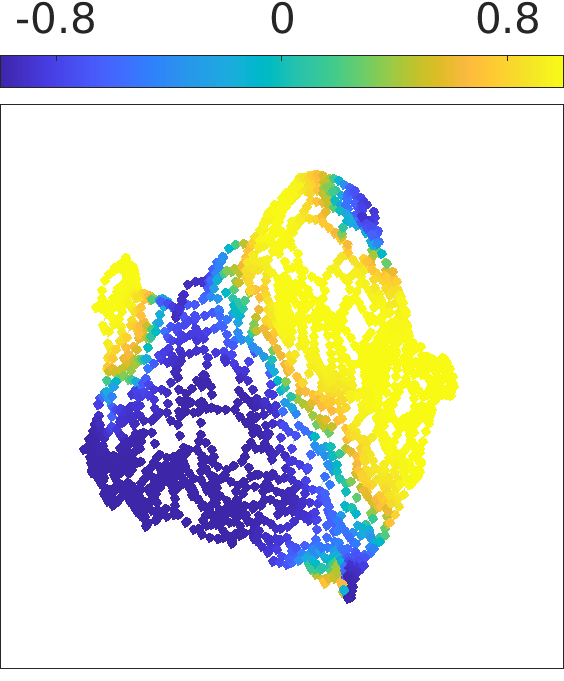}
	\includegraphics[width=0.16\textwidth]{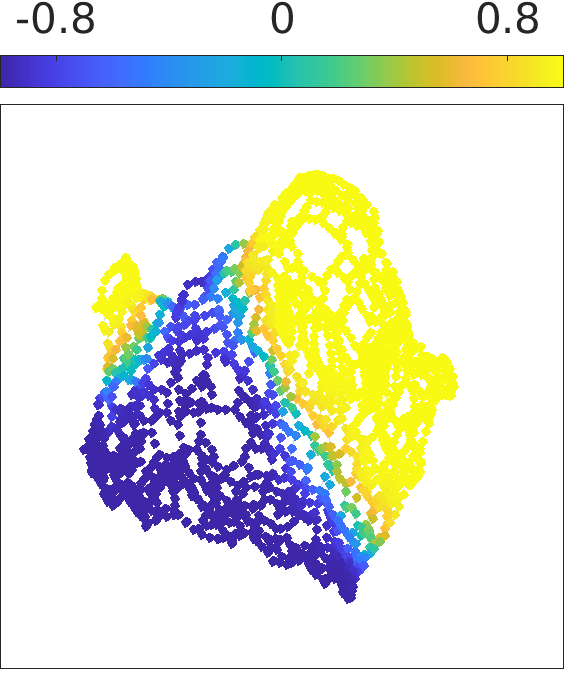}
	\caption{Allen--Cahn simulation result on the minnesota network.
		The parameters are chosen $\epsilon=0.05, D=5 \cdot 10^3$.
		The time interval is $[0,1]$ with a time step size of $h_i=0.01$.
		From left to right, the corresponding times are $t=0, t=0.1, t=0.3, t=0.35, t=0.5, t=1$.}\label{fig:graph_minnesota_AC_example_solution}
\end{figure}

\begin{figure}
	\includegraphics[width=0.16\textwidth]{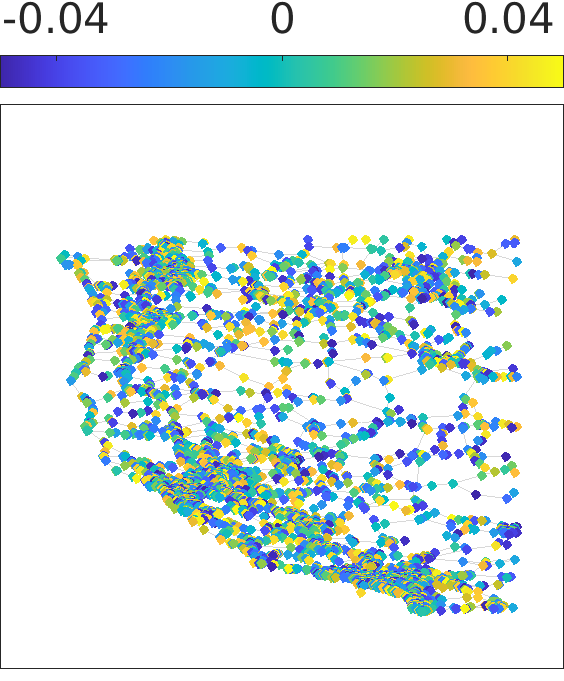}
	\includegraphics[width=0.16\textwidth]{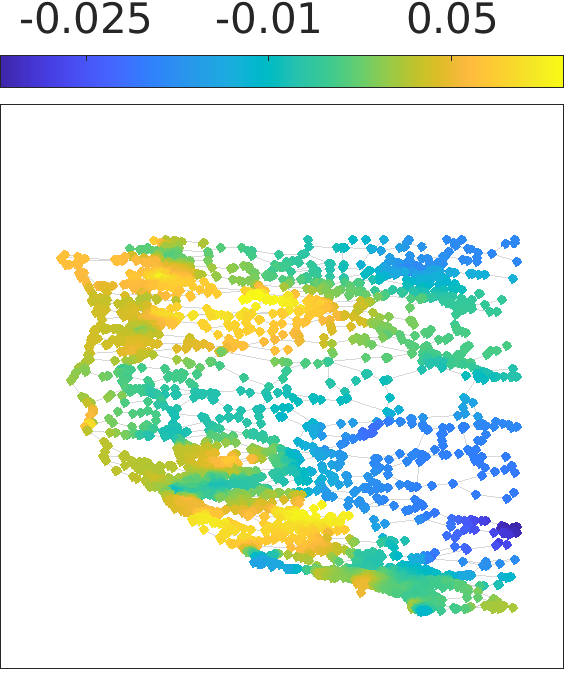}
	\includegraphics[width=0.16\textwidth]{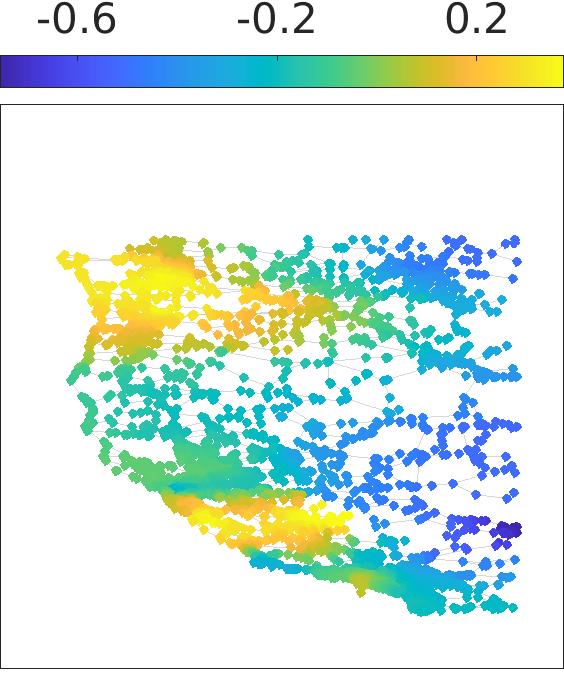}
	\includegraphics[width=0.16\textwidth]{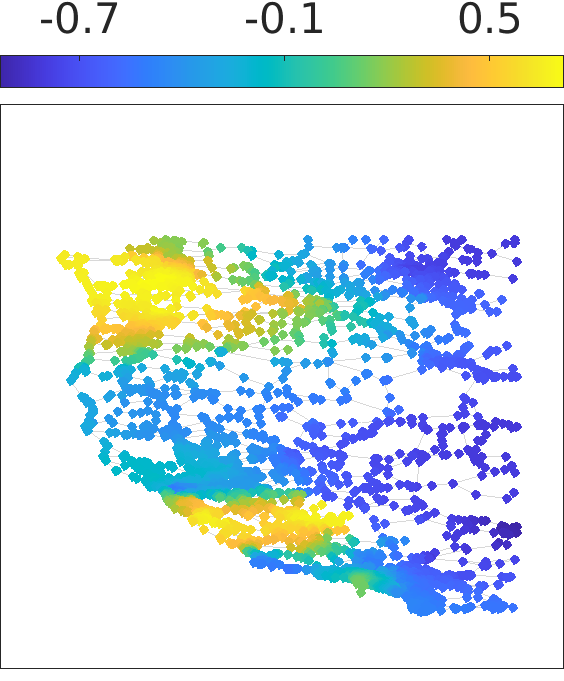}
	\includegraphics[width=0.16\textwidth]{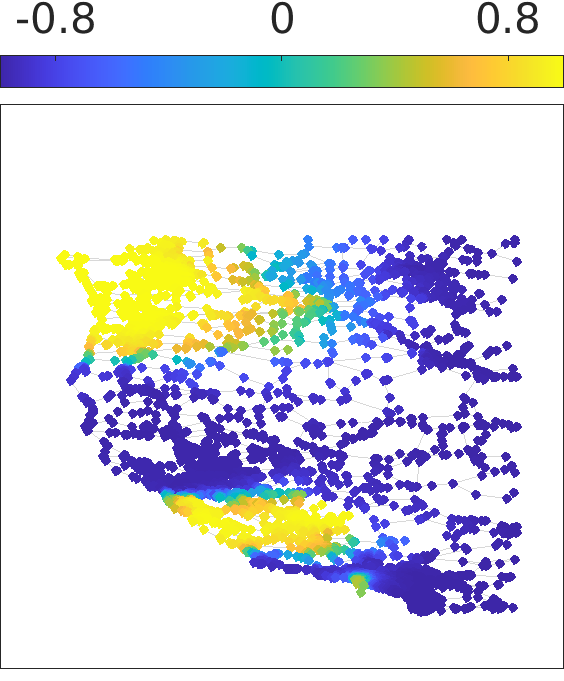}
	\includegraphics[width=0.16\textwidth]{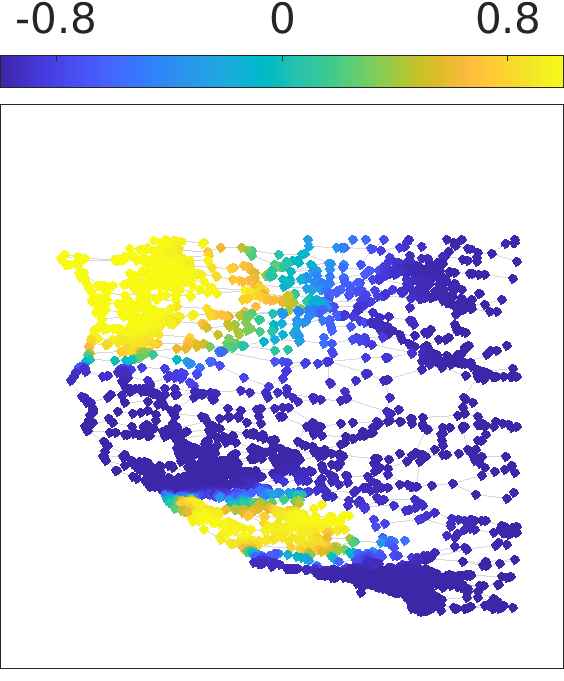}
	\caption{Allen--Cahn simulation result on the usroads (subset) network.
		The parameters are chosen $\epsilon=0.05, D=5 \cdot 10^4$.
		The time interval is $[0,1]$ with a time step size of $h_i=0.01$.
		From left to right, the corresponding times are $t=0, t=0.1, t=0.3, t=0.35, t=0.5, t=1$.}\label{fig:graph_usroads_AC_example_solution}
\end{figure}

As discussed in \Cref{sec:odes}, we can also solve the previously considered semi-linear parabolic PDEs on graphs or networks by using the (unnormalized) graph Laplacian as discrete linear differential operator.
Such problems (with an additional data fidelity term), for example, arise in semi-supervised learning techniques on graphs \cite{bertozzi2012diffuse,bertozzi2016diffuse,budd2021classification,bergermann2021semi}.

\begin{figure}
	\centering
	\subfloat[Krylov iteration numbers]{
		\includegraphics[width=.42\textwidth]{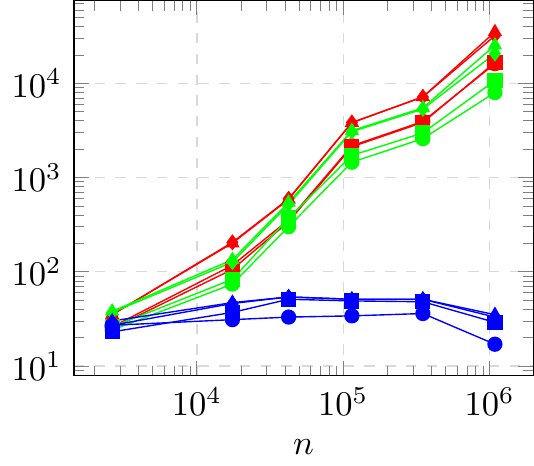}
	}
	\hfill
	\subfloat[Runtime in seconds]{
		\includegraphics[width=.42\textwidth]{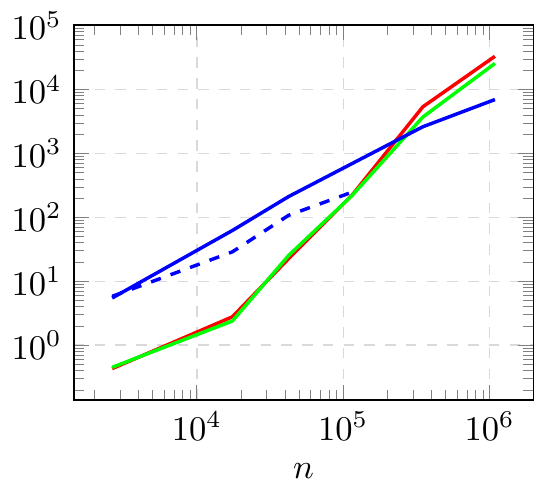}
	}
	\vspace{10pt}
	\includegraphics[width=.99\textwidth]{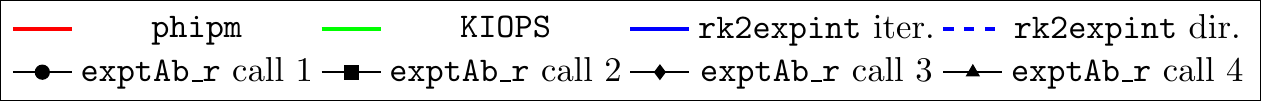}
	\vspace{-10pt}
	\caption{Average Krylov iteration numbers per evaluation of $e^{h_i \widetilde{\bm{A}}}\widetilde{\bm{c}}$ and total runtimes of solving the graph Allen--Cahn equation on the networks summarized in \Cref{tab:networks_AC}.
		We use the \texttt{Krogstad4} exponential Runge--Kutta integrator, which requires four evaluations of $e^{h_i \widetilde{\bm{A}}}\widetilde{\bm{c}}$ per time step.
		For \texttt{rk2expint}, we use complex-valued $(70,60)$ \texttt{RKFIT} poles fitted on the interval $[0,10^8]$ and report runtimes of direct and preconditioned iterative linear system solves.}\label{fig:iter_runtime_graph_AC}
\end{figure}

The major difference to the previously considered finite difference discretization of the Laplacian on continuous domains is that (unless appropriately weighted) networks do usually not contain spatial information.
In order to produce the usual patterns on the full network scale, we consider the scaled formulation
\begin{equation}\label{eq:allen_cahn_network}
\frac{\partial u}{\partial t} = \epsilon D \Delta u + \frac{1}{\epsilon} (u - u^3), \quad \epsilon,D\in\R,
\end{equation}
in the graph setting, which allows to trade off the linear diffusion and non-linear reaction parts of the equation.
In \Cref{tab:networks_AC}, we summarize weighted and unweighted undirected example networks downloaded from the SuiteSparse Matrix Collection\footnote{\url{https://sparse.tamu.edu/}}.
We take the largest connected component of each network to construct their unnormalized graph Laplacians.
Additionally, we choose the interface parameter $\epsilon=0.05$ for all networks and adjust the diffusion constant $D$ such that pattern formation on the full network scale is obtained.

\Cref{fig:graph_minnesota_AC_example_solution} and \Cref{fig:graph_usroads_AC_example_solution} show example trajectories of solutions to the graph Allen--Cahn equation on the minnesota and usroads (subset) networks, respectively.
Overall, we found that compared to discretized continuous domains, the same equation is more difficult to solve on graphs in the sense that smaller time steps and more rational Krylov poles are required to obtain accurate ODE solutions.
Throughout our experiments, we used complex-valued \texttt{RKFIT} poles of type $(70,60)$ fitted on the interval $[0,10^8]$.
For a time step size of $h_i=0.05$, we perform the usual comparison of average Krylov iteration numbers per time step and total runtimes and report the results in \Cref{fig:iter_runtime_graph_AC}.
The same general qualitative observations made in the discretized continuous setting hold true, i.e., approximately constant rational Krylov iteration numbers and a near-linear runtime dependence on $\|h_i\widetilde{\bm{A}}\|_2$, i.e., the problem size.
However, we observe a higher variation of rational Krylov iteration numbers across different networks, which is presumably related to varying network structures.
Furthermore, a less uniform increase in polynomial Krylov iteration numbers is observed, which is caused by the less uniform growth in the graph spectra reported in \Cref{tab:networks_AC}.

\subsection{Gierer--Meinhardt equations on networks}

Finally, we briefly comment on the Gierer--Meinhardt equations \cref{eq:pde_gm_a,eq:pde_gm_h} on networks.
While the qualitative behavior of the solutions is very different, the formation of Turing patterns in activator-inhibitor systems such as the Gierer--Meinhardt equations on scale-free networks has been shown to provide insights into biological networks such as cellular networks,  \cite{nakao2010turing}.
We illustrate an exemplary pattern formation process on networks in \Cref{fig:graph_brightkite_GM_example_solution} at the example of the largest connected component of the undirected scale-free loc-Brightkite network from the SuiteSparse Matrix Collection\footnote{\url{https://sparse.tamu.edu/}}.
Although we observe a differentiation into concentration-rich and -low groups, ordered periodic patterns can not be identified due to the lack of spatial relations between the nodes.

\begin{figure}
	\includegraphics[width=0.16\textwidth]{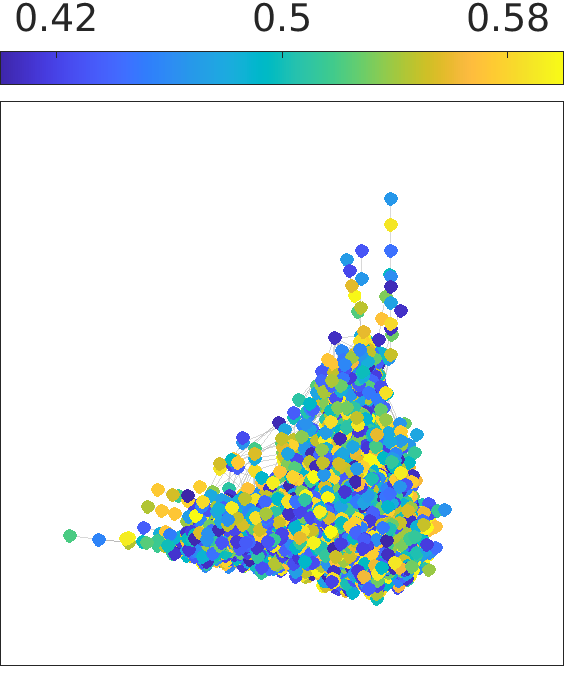}
	\includegraphics[width=0.16\textwidth]{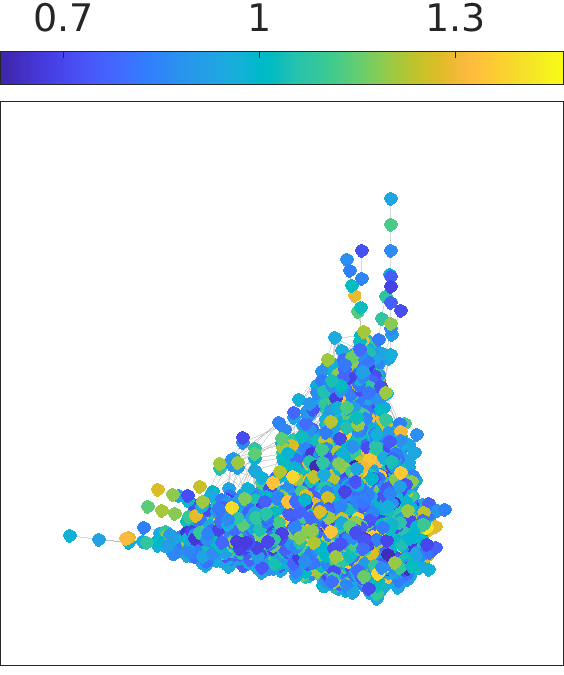}
	\includegraphics[width=0.16\textwidth]{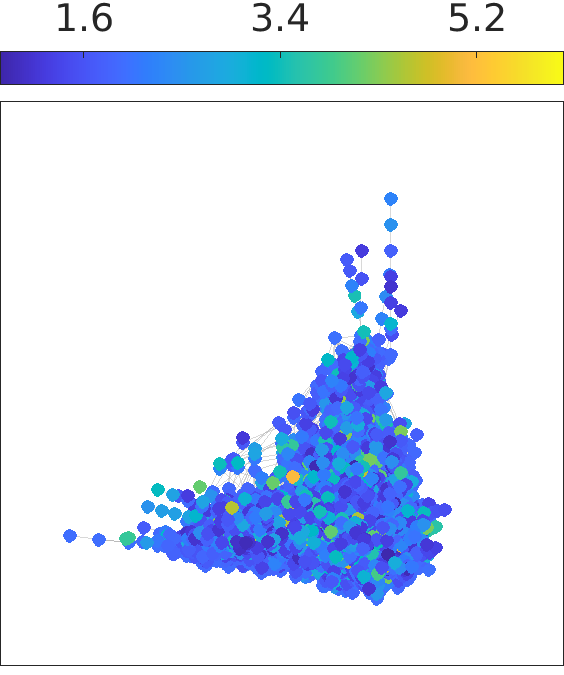}
	\includegraphics[width=0.16\textwidth]{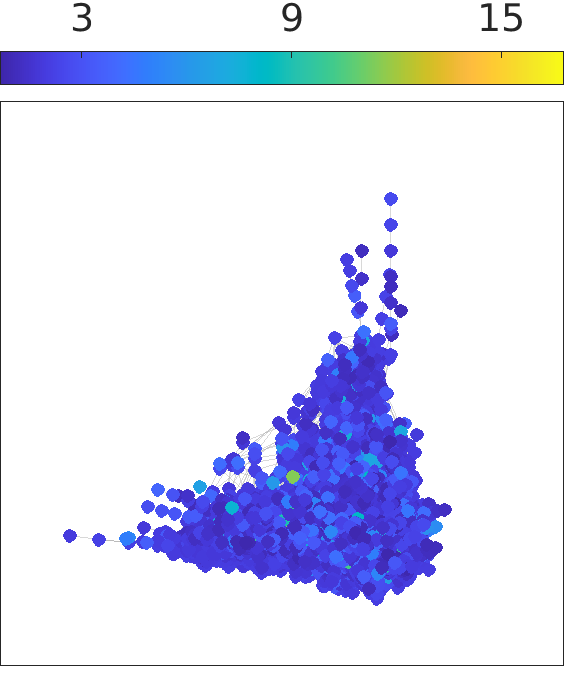}
	\includegraphics[width=0.16\textwidth]{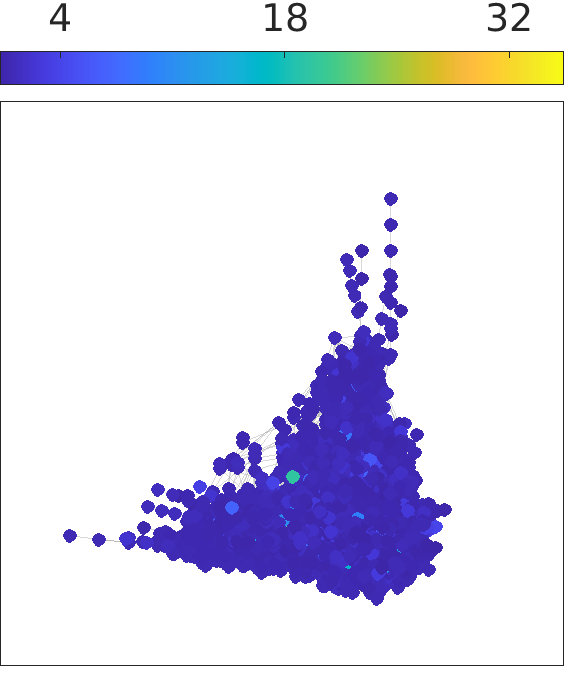}
	\includegraphics[width=0.16\textwidth]{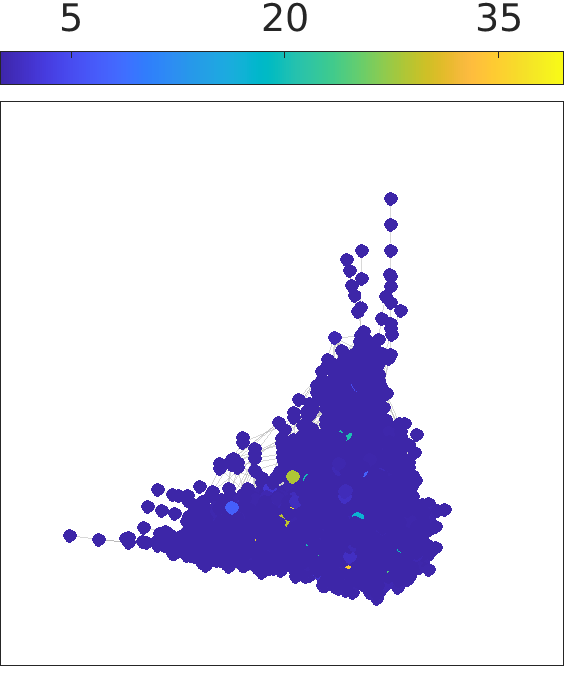}
	
	\vspace{2pt}
	
	\includegraphics[width=0.16\textwidth]{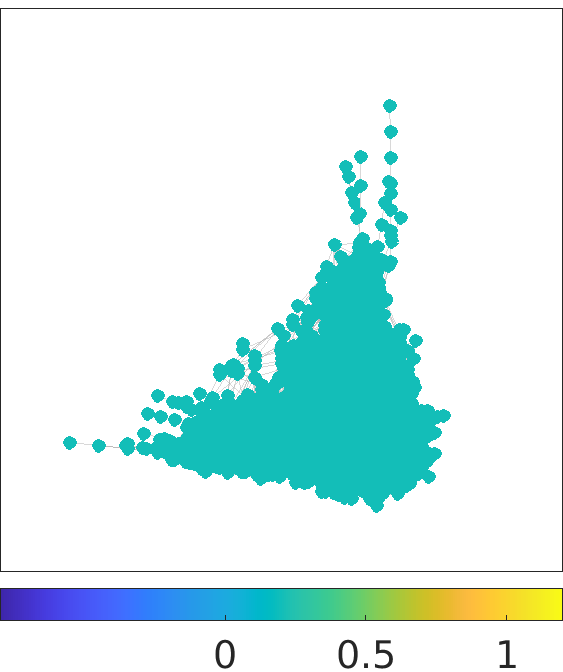}
	\includegraphics[width=0.16\textwidth]{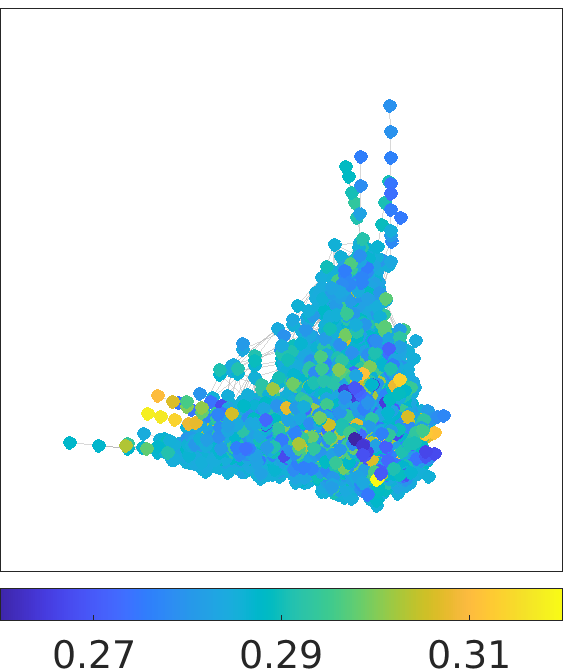}
	\includegraphics[width=0.16\textwidth]{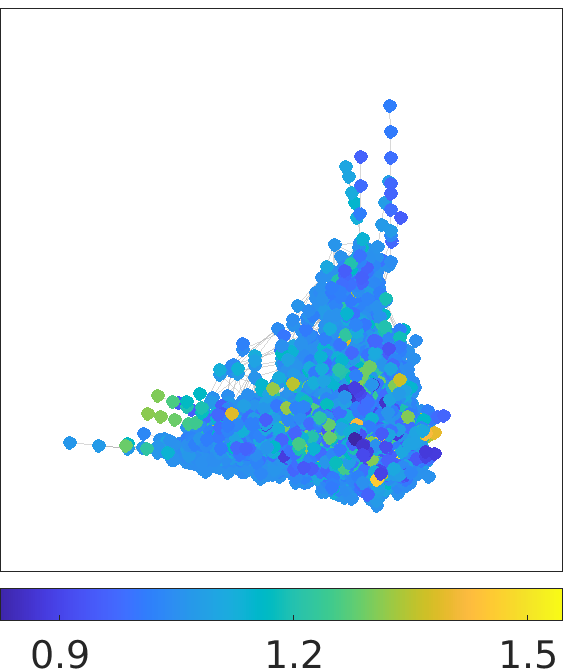}
	\includegraphics[width=0.16\textwidth]{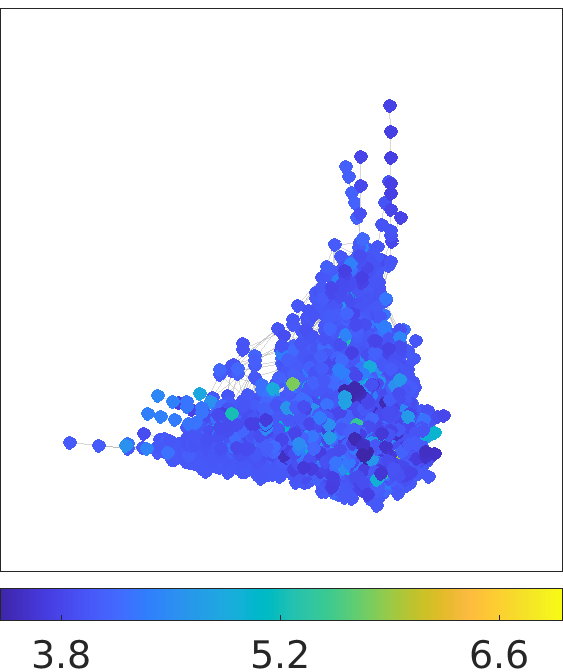}
	\includegraphics[width=0.16\textwidth]{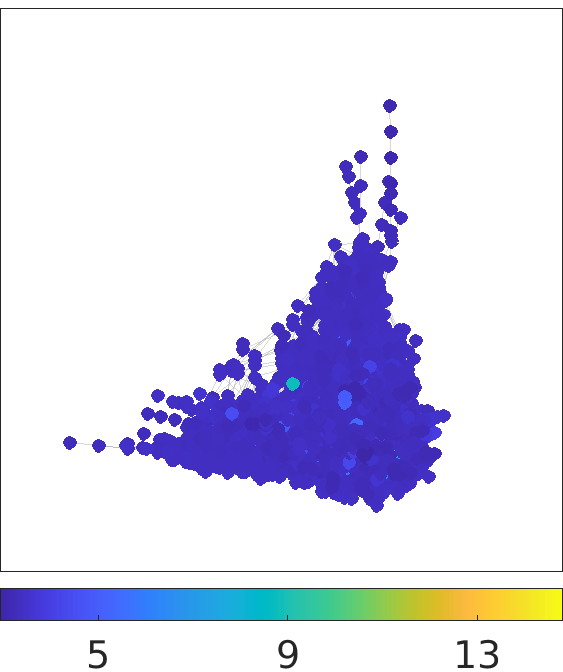}
	\includegraphics[width=0.16\textwidth]{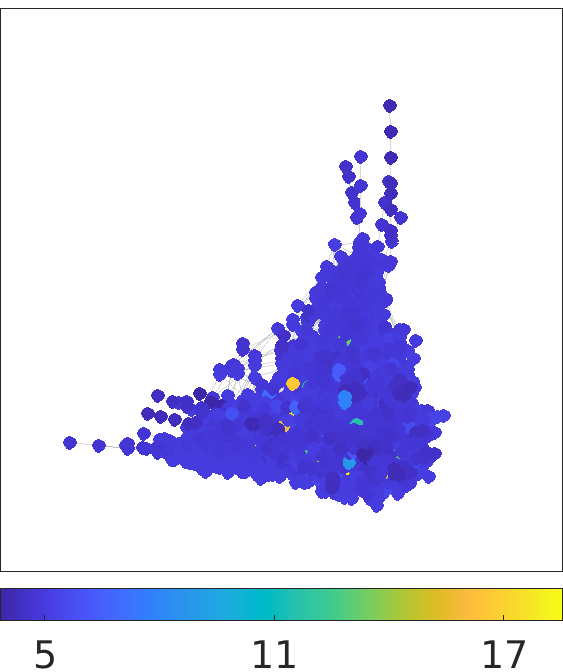}
	\caption{Gierer--Meinhardt simulation result on the loc-Brightkite network.
		The top row shows activator and the bottom row inhibitor concentrations.
		Initial activator concentrations are random while initial inhibitor concentrations are constant.
		The parameters are chosen $D_a=10, D_h=1\,000, p=\mu=p'=\nu = 8$ over the time interval $[0,1]$ with a time step size of $h_i=0.01$.
		From left to right, the corresponding times are $t=0, t=0.05, t=0.1, t=0.2, t=0.4, t=1$.}\label{fig:graph_brightkite_GM_example_solution}
\end{figure}

Numerically, we make the same observations as for the graph Allen--Cahn equation, i.e., smaller time step sizes and more rational Krylov poles are required for accurate ODE solutions.
Apart from this, approximately constant rational Krylov iteration numbers as well as a near-linear runtime of \texttt{rk2expint} are confirmed.

\section{Conclusion and outlook}\label{sec:conclusion}

This work presents an approach to apply adaptive rational Krylov methods to the efficient evaluation of exponential Runge--Kutta integrators used to solve large stiff systems of ODEs.
Numerical experiments confirm approximately constant rational Krylov iteration numbers independently of the problem size, the time step size, and the spectrum of the discrete linear differential operators.
This leads to a near-linear scaling of the runtime that can not be obtained by methods based on polynomial Krylov methods.

While we focused on real symmetric discrete linear differential operators, our approach should be extendable to more general nonsymmetric or complex-valued problems \cref{eq:ode_system}.
However, as in this situation the approximation domain for the exponential function is generally complex-valued, different pole selection strategies are required.
Furthermore, our method may be applicable to exponential Rosenbrock or EPIRK integrators in situations where the spectra of the local linearizations of general right-hand sides can be bounded.

\section*{Acknowledgments}

We thank Oliver Ernst, Stefan G\"uttel, and John Pearson for helpful hints and discussions.

\end{document}